\documentclass{amsart}
\usepackage{amssymb}
\usepackage{latexsym}
\usepackage{amsmath}
\usepackage{euscript} 
\usepackage{amsmath,amssymb,amsthm,graphicx,epstopdf,mathrsfs,url, color}

      \def\dR{{\mathbb R}}

\def\bm\chi{\mbox{\boldmath$\chi$}}

\def\min{{\rm min\,}}
\def\max{{\rm max\,}}
\def\ker{{\rm ker\,}}
\def\ran{{\rm ran\,}}

\def\dom{{\rm dom\,}}

\def\dim{{\rm dim\,}}

\let\xker=\ker \def\ker{{\xker\,}}

\def\senki{{\lbrack\negthinspace [\bot ]\negthinspace\rbrack}}
\def\senki+{{\lbrack\negthinspace [+] \negthinspace\rbrack}}

\newtheorem{theorem}{Theorem}[section]
\newtheorem{proposition}[theorem]{Proposition}
\newtheorem{corollary}[theorem]{Corollary}
\newtheorem{lemma}[theorem]{Lemma}
\theoremstyle{definition}

\newtheorem{definition}[theorem]{Definition}
\newtheorem{remark}[theorem]{Remark}

\numberwithin{equation}{section}

\begin{document}

\title[Single layer boundary integral operator for the Dirac equation]{On the single layer boundary integral operator for the Dirac equation}

\author[M.~Holzmann]{Markus Holzmann}

\address{Institut f\"ur Angewandte Mathematik\\
Technische Universit\"at Graz \\
Steyrergasse 30\\
8010 Graz \\
Austria}
\email{holzmann@math.tugraz.at}


\begin{abstract}
  This paper is devoted to the analysis of the single layer boundary integral operator $\mathcal{C}_z$ for the Dirac equation in the two- and three-dimensional situation. The map $\mathcal{C}_z$ is the strongly singular integral operator having the integral kernel of the resolvent of the free Dirac operator $A_0$ and $z$ belongs to the resolvent set of $A_0$. In the case of smooth boundaries fine mapping properties and a decomposition of $\mathcal{C}_z$ in a 'positive' and 'negative' part are analyzed. The obtained results can be applied in the treatment of Dirac operators with singular electrostatic, Lorentz scalar, and anomalous magnetic interactions that are combined in a critical way.
\end{abstract}

\maketitle

\section{Introduction}

In the analysis of boundary value and transmission problems for partial differential equations associated potential and  boundary integral operators often play an important role. These objects are well understood for uniformly elliptic second order differential expressions, cf. the monograph \cite{M00} and the references therein. We remark that the transmission problems are closely related to differential operators with singular interactions like, e.g., $\delta$-potentials. In the recent years it turned out that similar objects are also of importance in the study of boundary value and transmission problems for the Dirac equation, which is the relativistic counterpart of the Laplace equation and for which the associated differential expression is of first order and lacks ellipticity. While the potential operator is sufficiently well investigated, the subtle examination of the associated boundary integral operator is less complete, see the review paper \cite{BHSS22} and the references therein. It is the main goal in this article to make some further contributions to this study that are, in particular, necessary to examine Dirac operators with singular $\delta$-potentials, where the involved parameters are combined in a critical way.

To introduce the problem setting in  a more detailed way, consider first a formally symmetric and strongly elliptic second order partial differential operator $\mathcal{P}$ in $\mathbb{R}^q$, $q \geq 2$, let $P_0$ be the self-adjoint realization of $\mathcal{P}$ defined on $H^2(\mathbb{R}^q)$, and let $E_z(x,y)$ be the integral kernel of $(P_0 - z)^{-1}$, $z \in \rho(P_0)$. Moreover, let $\Omega \subset \mathbb{R}^q$ be a sufficiently smooth domain with unit normal vector field $\nu$ that is pointing outwards of $\Omega$ and let $\mathcal{B}_\nu$ be the conormal derivative at $\partial \Omega$ associated with $\mathcal{P}$. Then, the single layer potential $SL(z)$ and the double layer potential $DL(z)$ applied to a sufficiently smooth function $\varphi: \partial \Omega \rightarrow \mathbb{C}$ evaluated at  $x \in \mathbb{R}^q \setminus \partial \Omega$ are
\begin{equation*}
  \begin{split}
    SL(z) \varphi(x) &= \int_{\partial \Omega} E_z(x,y) \varphi(y) \textup{d} \sigma(y),
    \quad 
    DL(z) \varphi(x) = \int_{\partial \Omega} \partial_{\mathcal{B}_\nu, y} E_z(x,y) \varphi(y) \textup{d} \sigma(y).
  \end{split}
\end{equation*}
It is known that all solutions of the partial differential equation 
\begin{equation} \label{diff_eq}
  (\mathcal{P}-z) f = g \quad \text{in} \quad \mathbb{R}^q \setminus \partial \Omega
\end{equation}
can be described with the help of $SL(z)$ and $DL(z)$. To analyze boundary value or transmission problems associated with the above equation, it is common to employ the Dirichlet trace operator $\gamma_D$ and $\mathcal{B}_\nu$ to $SL(z)$ and $DL(z)$. Of particular interest is the single layer boundary integral operator $\mathcal{S}(z) := \gamma_D SL(z)$. Its properties are closely related to the solvability of the Dirichlet boundary value problem for~\eqref{diff_eq}, and relevant properties, as, e.g., its ellipticity, are inherited from those of $\mathcal{P}$, cf. \cite[Chapter~7]{M00}. Moreover,  a good understanding of $\mathcal{S}$ can be useful for a thorough understanding of different problems related to $\mathcal{P}$ in an operator theoretic language, for instance to show the self-adjointness and compute spectral properties of operators associated with boundary value or transmission problems for $\mathcal{P}$ including perturbations of $P_0$ by singular interactions, cf. \cite{BLL13, BEKS94, HU20, M00}.

It is the main goal in this paper to provide some properties of the counterpart of the single layer boundary integral operator for the Dirac operator in the two- and three-dimensional situation. For $m \geq 0$ the free Dirac operator in dimension two is given by
\begin{equation} \label{def_A_0_2d}
  A_{0} f = -i \sigma_1 \partial_1 f - i \sigma_2 \partial_2 f + m \sigma_3 f, \quad \dom A_{0} = H^1(\mathbb{R}^2; \mathbb{C}^2),
\end{equation}
where $\sigma_1, \sigma_2, \sigma_3 \in \mathbb{C}^{2 \times 2}$ are the Pauli spin matrices defined in~\eqref{Dirac_matrices1}, and in dimension three it is 
\begin{equation} \label{def_A_0_3d}
  A_{0} f = -i \alpha_1 \partial_1 f - i \alpha_2 \partial_2 f - i \alpha_3 \partial_3 f + m \beta f, \quad \dom A_{0} = H^1(\mathbb{R}^3; \mathbb{C}^4),
\end{equation}
where $\alpha_1, \alpha_2, \alpha_3, \beta \in \mathbb{C}^{4 \times 4}$ are the Dirac matrices in~\eqref{Dirac_matrices2}. Here, we used the notation $H^k(\mathbb{R}^q; \mathbb{C}^l) = H^k(\mathbb{R}^q) \otimes \mathbb{C}^l$ for the $L^2$-based Sobolev space of $k$ times weakly differentiable vector valued functions. The free Dirac operator  is used to describe the propagation of a spin $\frac{1}{2}$ particle in vacuum taking effects of the special theory of relativity into account \cite{T92}. Moreover, the two-dimensional Dirac operator appears in the mathematical description of graphene \cite{AB08}. It is known that $A_{0}$ is self-adjoint and that its spectrum is 
\begin{equation*}
  \sigma(A_{0}) = (-\infty, -m] \cup [m, \infty),
\end{equation*}
cf. \cite{T92}. This shows, in particular, that $A_{0}$ is not semi-bounded from above or below. For $z \in \rho(A_{0}) = \mathbb{C} \setminus ((-\infty, -m] \cup [m, \infty))$ the resolvent of $A_{0}$ can be expressed via the convolution with a function $G_{z,q}$ given in~\eqref{def_G_lambda} below, i.e. as an integral operator. With the help of this function we can formally introduce for a smooth and closed curve $\Sigma \subset \mathbb{R}^2$ or a smooth and closed surface $\Sigma \subset \mathbb{R}^3$ the boundary integral operator acting on sufficiently smooth functions $\varphi: \Sigma \rightarrow \mathbb{C}^N$ as
\begin{equation} \label{def_C_z_intro}
  \mathcal{C}_z \varphi(x) := \lim_{\varepsilon \searrow 0} \int_{\Sigma \setminus B(x, \varepsilon)} 
  G_{z,q}(x-y) \varphi(y) \textup{d}\sigma(y), \quad x \in \Sigma,
\end{equation}
where $B(x, \varepsilon)$ is the ball centered at $x$ with radius $\varepsilon$ and $N=2$ for $q=2$ and $N=4$ for $q=3$.
It is known that $\mathcal{C}_z$ gives rise to a bounded operator in $L^2(\Sigma; \mathbb{C}^N)$, but in a similar way as $A_{0}$ also $\mathcal{C}_z$ lacks ellipticity \cite{AMV14, AGHM01, BHOP20}. 

The operator $\mathcal{C}_z$ plays an important role in the analysis of boundary value and transmission problems for the Dirac equation and motivated by this $\mathcal{C}_z$ was studied intensively in the recent years, cf. \cite{AMV14, AMSV22, BEHL18, BHM20, BHOP20, BHSS22, Ben21, Ben22, BBZ22, OV16}. In the present paper  this study is continued and, in particular, detailed mapping properties and a more detailed analysis of the contribution of $\mathcal{C}_z$ associated with the 'positive' and the 'negative' part of its spectrum are provided. More precisely, as a consequence of the main results, it is shown in dimension $q=2$ in Corollary~\ref{corollary_positive_negative_2d} that 
\begin{equation*}
  \mathcal{C}_z = \frac{1}{2} (\sigma_1 \nu_1 + \sigma_2 \nu_2) V^* \begin{pmatrix} 0 & P_- - P_+ \\ P_- - P_+ & 0 \end{pmatrix} V (\sigma_1 \nu_1 + \sigma_2  \nu_2) + \mathcal{K},
\end{equation*}
and similarly in dimension $q=3$ in Corollary~\ref{corollary_positive_negative_3d} that
\begin{equation*}
  \mathcal{C}_z = \frac{1}{2} (\alpha_1 \nu_1 + \alpha_2 \nu_2 + \alpha_3 \nu_3) V^* \begin{pmatrix} 0 & P_- - P_+ \\ P_- - P_+ & 0 \end{pmatrix} V (\alpha_1 \nu_1 + \alpha_2 \nu_2 + \alpha_3 \nu_3) + \mathcal{K},
\end{equation*}
where $\nu = (\nu_1, \dots, \nu_q)$ is the unit normal vector on $\Sigma$ pointing outwards of the bounded domain with boundary $\Sigma$, $V$ is a unitary matrix-valued function defined in~\eqref{def_V_2d} and~\eqref{def_V_3d} below, $\mathcal{K}$ is an operator with good mapping properties between different Sobolev spaces (and hence, $\mathcal{K}$ is compact in $H^s(\Sigma; \mathbb{C}^N)$ for $s \in [-1, 1]$), and $P_\pm$ are self-adjoint operators in $L^2(\Sigma; \mathbb{C}^{N/2})$ that satisfy 
\begin{equation} \label{properties_P_pm}
  P_+ + P_- = I, \quad \dim \ran P_\pm = \infty, \quad P_\pm^2 = P_\pm, \quad \text{and} \quad P_\pm P_\mp = 0,
\end{equation}
i.e. $P_\pm$ are orthogonal projections in $L^2(\Sigma; \mathbb{C}^{N/2})$. While for dimension $q=2$ the operators $P_\pm$ are also orthogonal projections in $H^s(\Sigma; \mathbb{C})$, $s \in \mathbb{R}$, that satisfy~\eqref{properties_P_pm}, in dimension $q=3$ we prove that $P_\pm$ give rise to bounded maps in $H^s(\Sigma; \mathbb{C}^2)$, $s \in [-1, 1]$, that fulfil~\eqref{properties_P_pm} in $H^s(\Sigma; \mathbb{C}^2)$.
We remark that in dimension $q=3$ similar, but not exactly the same decompositions of $\mathcal{C}_z$ were considered in \cite{AMSV22, OV16} with skew projections onto Hardy spaces on $\Sigma$; cf. Remark~\ref{remark_Hardy_spaces} for details. However, we believe that the above representations are of interest, as the projections in \cite{AMSV22, OV16} are not self-adjoint in $L^2(\Sigma; \mathbb{C}^{N/2})$ unless $\Sigma$ is a sphere, but which is useful in some applications. One possible application of this fact is shown in the very recent paper \cite{BHSS22}, where Dirac operators perturbed by singular potentials involving electrostatic, Lorentz scalar, and anomalous magnetic interactions combined in a critical way are studied. 

It should be remarked that the above mentioned results are proved in different ways in dimension two and three. In dimension two $\mathcal{C}_z$ is closely related to the Cauchy transform on $\Sigma$, for its analysis we follow ideas from \cite{BHOP20} and apply the theory of periodic pseudodifferential operators. In dimension three $\mathcal{C}_z$ is closely related to the Riesz transform on $\Sigma$ and results on pseudo-homogeneous kernels from \cite{N01} are employed to show similar results as in dimension two. In particular, the latter approach would allow to weaken the geometric assumptions on $\Sigma$ and to consider also higher space dimensions. We remark that in the recent paper \cite{BP22}, which was written independently of this article, the operator $\mathcal{C}_z$ is studied in the three-dimensional situation with the help of pseudodifferential techniques and results that are related to the ones in this paper are recovered there as well.

The paper is organized as follows. In Section~\ref{section_C_z} we rigorously introduce the operator $\mathcal{C}_z$ formally given by~\eqref{def_C_z_intro} and discuss its basic properties. Then, in Section~\ref{section_2d} we recall some results about periodic pseudodifferential operators and do a more detailed analysis of $\mathcal{C}_z$ in the two-dimensional case. Finally, in Section~\ref{section_3d} we revise basic notions and results on pseudo-homogeneous kernels and use them to investigate $\mathcal{C}_z$ in dimension three.

\subsection*{Acknowledgement.}
The author is grateful to the anonymous referee for helpful suggestions to improve the paper. He also thanks Jussi Behrndt, Dale Frymark, Christian Stelzer, and Georg Stenzel for fruitful discussions.
Moreover, he gratefully acknowledges financial support by the Austrian Science Fund (FWF): P33568-N. This publication is based upon work from COST Action CA 18232 MAT-DYN-NET, supported by COST (European Cooperation in Science and Technology), www.cost.eu.

\section{Notations and basic properties of $\mathcal{C}_z$} \label{section_C_z}

Throughout this paper, let $q \in \{ 2, 3 \}$ be the space dimension and define the number $N = N(q)$ by $N(2)=2$ and $N(3)=4$. Moreover, we assume that $\Omega \subset \mathbb{R}^q$ is a bounded and simply connected domain with $C^\infty$-smooth boundary $\Sigma := \partial \Omega$ and unit normal vector field $\nu$ which is pointing outwards of $\Omega$. Let 
\begin{equation} \label{Dirac_matrices1}
\sigma_1 = \left( \begin{array}{cc}
0 & 1\\                                              
1 & 0 \\                                            
\end{array}\right), \quad \sigma_2 = \left( \begin{array}{cc}
0 & -i\\                                              
i & 0 \\                                            
\end{array}\right), \quad \sigma_3 = \left( \begin{array}{cc}
1 & 0\\                                              
0 & -1 \\                                            
\end{array}\right) 
\end{equation}
be the Pauli spin matrices and define the $4\times 4$ Dirac matrices by
\begin{equation} \label{Dirac_matrices2}
\alpha_j = \left( \begin{array}{cc}
0 & \sigma_j \\                                              
\sigma_j & 0 \\                                            
\end{array}\right), \quad j \in \{1,2,3\}, \quad  \beta = \left( \begin{array}{cc}
I_2 & 0 \\                                              
0 & -I_2 \\                                            
\end{array}\right),
\end{equation}
where $I_n$ is the $n \times n$-identity matrix. We will often use for $x = (x_1,x_2) \in \mathbb{C}^2$ the notation
\begin{equation*}
  \sigma \cdot x = \sigma_1 x_1 + \sigma_2 x_2
\end{equation*}
and for $x = (x_1, x_2, x_3) \in \mathbb{C}^3$
\begin{equation*}
  \alpha \cdot x = \alpha_1 x_1 + \alpha_2 x_2 + \alpha_3 x_3 \quad \text{and} \quad
  \sigma \cdot x = \sigma_1 x_1 + \sigma_2 x_2 + \sigma_3 x_3.
\end{equation*}
Next, we introduce the function $G_{z, q}$ evaluated at $x \in \mathbb{R}^q \setminus \{ 0 \}$ by
\begin{equation} \label{def_G_lambda}
\begin{split}
G_{z, 2}(x) &= \frac{\sqrt{z^2 - m^2}}{2 \pi} K_1 \big( - i \sqrt{z^2 - m^2} | x |\big)  \frac{(\sigma \cdot x)}{ | x | }  \\
& \qquad \qquad \qquad+ \frac{1}{2 \pi} K_0 \big(- i \sqrt{z^2 - m^2} | x |\big) \big( z I_2 + m \sigma_3\big), \\
G_{z,3}(x) &= \left( z I_4 + m \beta + \left( 1 - i \sqrt{z^2 - m^2} | x |\right) \frac{i (\alpha \cdot x)}{  | x |^2}  \right) \frac{1}{4 \pi | x |} e^{i \sqrt{z^2 - m^2} |x|},
\end{split}
\end{equation}
where we write $K_j$ for the modified Bessel functions of the second kind
and choose $\sqrt{w}$ for $w \in \mathbb{C} \setminus [0, \infty)$ such that $\textup{Im} \sqrt{w} > 0$. It is well-known that $G_{z, q}$ is the integral kernel of the resolvent of the free Dirac operator $A_0$ in~\eqref{def_A_0_2d} \&~\eqref{def_A_0_3d}; cf. \cite{BHOP20, BHSS22, T92}.

Now, we are prepared to rigorously introduce and discuss the properties of the operator $\mathcal{C}_z$, $z \in \rho(A_0)=(-\infty, -m] \cup [m, \infty)$, formally given by~\eqref{def_C_z_intro}, i.e. we consider now the strongly singular integral operator
\begin{equation} \label{def_C_lambda}
  \mathcal{C}_z \varphi(x) := \lim_{\varepsilon \searrow 0} \int_{\Sigma \setminus B(x, \varepsilon)} 
  G_{z,q}(x-y) \varphi(y) \textup{d}\sigma(y), \quad
  \varphi \in C^\infty(\Sigma; \mathbb{C}^N),~x \in \Sigma;
\end{equation}
here $B(x, \varepsilon)$ is the ball of radius $\varepsilon$ centered at $x$.
The basic properties of $\mathcal{C}_z$ are summarized in the following proposition; in the proof we mostly refer to \cite{BHSS22}, but parts of the results were shown before in \cite{BHOP20} in dimension $q=2$ and in \cite{AMV14, AMV15, BH20, BHM20, BBZ22, OV16} in dimension $q=3$. Below, $H^s(\Sigma)$, $s \in \mathbb{R}$, are the Sobolev spaces on $\Sigma$ defined as in \cite{M00} and we denote by $( \cdot, \cdot )_{H^s(\Sigma; \mathbb{C}^N) \times H^{-s}(\Sigma; \mathbb{C}^N)}$ the sesquilinear duality product in $H^s(\Sigma; \mathbb{C}^N) \times H^{-s}(\Sigma; \mathbb{C}^N)$.

\begin{proposition} \label{proposition_C_z_basic_properties}
  For the operator $\mathcal{C}_z$, $z \in \rho(A_0)=\mathbb{C} \setminus ((-\infty, -m] \cup [m, \infty))$, introduced in~\eqref{def_C_lambda} and any $s \in \mathbb{R}$ the following holds:
  \begin{itemize}
    \item[(i)] The map $\mathcal{C}_z$ gives rise to a bounded operator in $H^s(\Sigma; \mathbb{C}^N)$.
    \item[(ii)] For $\varphi \in H^s(\Sigma; \mathbb{C}^N)$ and $\psi \in H^{-s}(\Sigma; \mathbb{C}^N)$ one has 
    \begin{equation*}
      ( \mathcal{C}_z \varphi, \psi )_{H^s(\Sigma; \mathbb{C}^N) \times H^{-s}(\Sigma; \mathbb{C}^N)} = ( \varphi, \mathcal{C}_{\overline{z}} \psi )_{H^s(\Sigma; \mathbb{C}^N) \times H^{-s}(\Sigma; \mathbb{C}^N)}.
    \end{equation*}
    In particular, for the realization of $\mathcal{C}_z$ in $L^2(\Sigma; \mathbb{C}^N)$ the relation $\mathcal{C}_z^* = \mathcal{C}_{\overline{z}}$ holds.
    \item[(iii)] For $q=2$ one has $-4 (\mathcal{C}_z (\sigma \cdot \nu))^2 = -4 ((\sigma \cdot \nu) \mathcal{C}_z)^2 = I_2$ and for $q=3$ the relation $-4 (\mathcal{C}_z (\alpha \cdot \nu))^2 = -4 ((\alpha \cdot \nu) \mathcal{C}_z)^2 = I_4$ holds.
  \end{itemize}
\end{proposition}
\begin{proof}
  The mapping properties of $\mathcal{C}_z$ in item~(i) follow from \cite[Proposition~3.3]{BHOP20} for $q=2$ and \cite[Theorem~4.1]{BBZ22} for $q=3$. 
  The claim in statement~(ii) is shown for $s \in [-\frac{1}{2}, \frac{1}{2}]$ in \cite[Proposition~4.4]{BHSS22} taking into account that the spaces $H^s_\alpha(\Sigma; \mathbb{C}^N)$ in \cite{BHSS22} coincide for $C^\infty$-smooth $\Sigma$ with $H^s(\Sigma; \mathbb{C}^N)$. Together with this a simple density and continuity argument yields the claim for  $s \notin [-\frac{1}{2}, \frac{1}{2}]$.
  Similarly, assertion~(iii) can be found for $s \in [-\frac{1}{2}, \frac{1}{2}]$ in \cite[equations~(4.6) and (4.24)]{BHSS22}, the  statement for general $s \in \mathbb{R}$ follows then by restriction for $s > \frac{1}{2}$ and duality for $s < -\frac{1}{2}$.
\end{proof}

\section{A more detailed analysis in dimension $q=2$} \label{section_2d}

In this section we investigate $\mathcal{C}_z$ in space dimension $q=2$ in a more detailed way. For this, we make use of the theory of periodic pseudodifferential operators. Their definition and some basic properties are recalled in Subsection~\ref{section_PPDO}. With the help of these results we analyze $\mathcal{C}_z$ in Subsection~\ref{section_C_z_2d}.

\subsection{Periodic pseudodifferential operators} \label{section_PPDO}

In this subsection we follow closely the short exposition on periodic pseudodifferential operators in \cite[Section~2.1]{BHOP20}, for a more comprehensive presentation see \cite{SV}. 
In order to define periodic pseudodifferential operators, some notations are necessary. We set $\mathbb{T} := \mathbb{R} / \mathbb{Z}$ and denote by $\mathcal{D}(\mathbb{T})$ the set of all $1$-periodic test functions and by $\mathcal{D}'(\mathbb{T})$ the set of all $1$-periodic distributions.  We will often make use of the functions $e_n \in \mathcal{D}(\mathbb{T})$ defined by
\begin{equation} \label{def_e_n}
  e_n(t) := e^{2 \pi i n t}, \qquad n \in \mathbb{Z}.
\end{equation}
For $f \in \mathcal{D}'(\mathbb{T})$ we introduce the Fourier coefficients by
\begin{equation} \label{def_Fourier_coeff}
  \widehat{f}(n) := \langle f, e_{-n} \rangle_{\mathcal{D}'(\mathbb{T}) \times \mathcal{D}(\mathbb{T})},
\end{equation}
where $\langle \cdot, \cdot \rangle_{\mathcal{D}'(\mathbb{T}) \times \mathcal{D}(\mathbb{T})}$ denotes the bilinear duality product in $\mathcal{D}'(\mathbb{T}) \times \mathcal{D}(\mathbb{T})$. Set 
\begin{equation*}
  \underline{n} := \begin{cases} 1, &\text{ if }n=0, \\ |n|, &\text{ if } n \in \mathbb{Z} \setminus \{ 0 \}. \end{cases}
\end{equation*}
With the help of the Fourier coefficients one can define for $s \in \mathbb{R}$ the Sobolev space $H^s(\mathbb{T})$ of order $s$ on $\mathbb{T}$ by
\begin{equation} \label{def_Sobolev_space_interval}
  H^s(\mathbb{T}) := \left\{ f \in \mathcal{D}'(\mathbb{T}): \sum_{n \in \mathbb{Z}} \underline{n}^{2 s} |\widehat{f}(n)|^2  < \infty \right\}
\end{equation}
and endow it with the natural inner product
\begin{equation*}
  (f, g)_{H^s(\mathbb{T})} := \sum_{n \in \mathbb{Z}} \underline{n}^{2 s} \widehat{f}(n) \overline{\widehat{g}(n)}, \qquad u, v \in H^s(\mathbb{T}).
\end{equation*}

The Fourier coefficients in~\eqref{def_Fourier_coeff} allow us also to establish periodic pseudodifferential operators.

\begin{definition}
  A linear map $A$ defined on $\mathcal{D}(\mathbb{T})$ is called periodic pseudodifferential operator of order $s \in \mathbb{R}$, if there exists a function $h: \mathbb{T} \times \mathbb{Z} \rightarrow \mathbb{C}$ such that the following holds:
  \begin{itemize}
    \item[(i)] For any fixed $n \in \mathbb{Z}$ one has $h(\cdot, n) \in \mathcal{D}(\mathbb{T})$.
    \item[(ii)] The action of $A$ is given by
    \begin{equation*}
      A f = \sum_{n \in \mathbb{Z}} h(\cdot, n) \widehat{f}(n) e_n
    \end{equation*}
    with the sum converging in $\mathcal{D}'(\mathbb{T})$.
    \item[(iii)] For all $p, q \in \mathbb{N}_0$ there exists $c_{p,q} > 0$ such that for all $n \in \mathbb{Z}$ and $t \in \mathbb{R}$
    \begin{equation*}
      \left| \left( \frac{\textup{d}^p}{\textup{d} t^p} \omega^q h\right) (t,n) \right| \leq c_{p,q} (1+|n|)^{s-q}
    \end{equation*}
    holds,
    where $\omega h(t,n) := h(t,n+1)-h(t,n)$.
  \end{itemize}
  The set of all periodic pseudodifferential operators of order $s$ is denoted by $\Psi^s$. Moreover, we set $\Psi^{-\infty} := \bigcap_{s \in \mathbb{R}} \Psi^s$.
\end{definition}

In the following let $\Sigma \subset \mathbb{R}^2$ be the boundary of a simply connected $C^\infty$-domain, let $\ell$ be the length of $\Sigma$, and let $\gamma: [0, \ell] \rightarrow \mathbb{R}^2$ be an arc-length parametrization of $\Sigma$ that is orientated in the counter clockwise way. Similarly as above, we denote by $\mathcal{D}(\Sigma)$ the set of all $C^\infty$-smooth functions defined on $\Sigma$ and by $\mathcal{D}'(\Sigma)$ the set of all distributions on $\Sigma$. In order to define periodic pseudodifferential operators on $\Sigma$, we introduce the mapping $U: \mathcal{D}'(\Sigma) \rightarrow \mathcal{D}'(\mathbb{T})$ by
\begin{equation} \label{def_U}
  U f(\varphi) := f \big( \ell^{-1} \varphi (\ell^{-1} \gamma^{-1} (\cdot )) \big), \qquad f \in \mathcal{D}'(\Sigma), ~\varphi \in \mathcal{D}(\mathbb{T}).
\end{equation}
If $f$ is a regular distribution, then $Uf$ is also a regular distribution generated by $Uf = f(\gamma(\ell \cdot))$. With the map $U$ one can translate  the Sobolev spaces $H^s(\mathbb{T})$ on $\mathbb{T}$ to Sobolev spaces on $\Sigma$ by
\begin{equation} \label{def_H_s}
  H^s(\Sigma) := \big\{ \varphi \in \mathcal{D}'(\Sigma): U \varphi \in H^s(\mathbb{T}) \big\}, \quad s \in \mathbb{R},
\end{equation}
and endow them with the inner product
\begin{equation*}
  (\varphi, \psi)_{H^s(\Sigma)} := (U \varphi, U \psi)_{H^s(\mathbb{T})}, \qquad \varphi, \psi \in H^s(\Sigma).
\end{equation*}
By this construction the map $U$ defined in~\eqref{def_U} is for any $s \in \mathbb{R}$ a unitary operator from $H^s(\Sigma)$ to $H^s(\mathbb{T})$ and the corresponding norm is equivalent to the norm in $H^s(\Sigma)$ defined as in \cite{M00}.

\begin{definition}
  A map $A$ defined on $\mathcal{D}(\Sigma)$ is called periodic pseudodifferential operator on $\Sigma$ of order $s \in \mathbb{R}$, if there exists $A_0 \in \Psi^s$ such that $A = U^{-1} A_0 U$.
  The set of all periodic pseudodifferential operators on $\Sigma$ of order $s$ is denoted by $\Psi^s_\Sigma$. Moreover, we set $\Psi^{-\infty}_\Sigma := \bigcap_{s \in \mathbb{R}} \Psi^s_\Sigma$.
\end{definition}

An important feature of periodic pseudodifferential operators is the fact that they are automatically bounded in Sobolev spaces \cite[Theorems~7.3.1 and~7.8.1]{SV}.

\begin{lemma} \label{lemma_PPDO_mapping_properties}
  Let $A \in \Psi^s_\Sigma$ and $B \in \Psi^t_\Sigma$ for some $s, t \in \mathbb{R}$. Then, the following is true:
  \begin{itemize}
    \item[(i)] For any $r \in \mathbb{R}$ the map $A$ can be extended to a well-defined and bounded operator $A: H^r(\Sigma) \rightarrow H^{r-s}(\Sigma)$.
    \item[(ii)] $A B \in \Psi^{s+t}_\Sigma$.
    \item[(iii)] $A B - B A \in \Psi^{s+t-1}_\Sigma$.
  \end{itemize}
\end{lemma}

There are three types of periodic pseudodifferential operators that we are going to use frequently. First, if $f \in \mathcal{D}(\Sigma)$, then one has for the associated multiplication operator
\begin{equation} \label{equation_mult_op}
  \big(\mathcal{D}(\Sigma) \ni \varphi \mapsto f \cdot \varphi\big) \in \Psi_\Sigma^0.
\end{equation}

Next, consider for a constant $c_\Lambda > 0$ and $t \in \mathbb{R}$ the map 
\begin{equation*}
  L^t f := \left(\frac{4 \pi}{\ell} \right)^{t/2} \sum_{n \in \mathbb{Z}} (c_\Lambda + |n|)^{t/2} \widehat{f}(n) e_n, \qquad f \in \mathcal{D}(\mathbb{T}).
\end{equation*}
Then it is not difficult to see that $L^{t_1} L^{t_2} = L^{t_1+t_2}$ holds for all $t_1,t_2 \in \mathbb{R}$ and $L^t \in \Psi^{t/2}$. Hence, we can define
\begin{equation} \label{def_Lambda_2d}
  \Lambda^t := U^{-1} L^t U \in \Psi^{t/2}_\Sigma.
\end{equation}
Taking the definition of $H^r(\Sigma)$ in~\eqref{def_Sobolev_space_interval} and~\eqref{def_H_s} into account it is not difficult to show that for any $r \in \mathbb{R}$ the map $\Lambda^t: H^r(\Sigma) \rightarrow H^{r-t/2}(\Sigma)$ is bijective and that the realization of $\Lambda^t$ as a possibly unbounded operator in $L^2(\Sigma)$ defined on $\dom \Lambda^t = H^{\tau}(\Sigma)$ with $\tau = \max\{ \frac{t}{2}, 0\}$ is self-adjoint, cf. \cite{BHOP20}.

The third periodic pseudodifferential operator that will play a crucial role in our analysis is a multiple of the  Cauchy transform on $\Sigma$. 
We define formally the integral operator $\mathcal{R}$ acting on $\varphi \in \mathcal{D}(\Sigma)$ by
\begin{equation} \label{def_R_2d}
  \mathcal{R} \varphi(x) = -\frac{2}{\pi} \lim_{\varepsilon \searrow 0} \int_{\Sigma \setminus B(x, \varepsilon)} \frac{\nu_1(y) + i \nu_2(y)}{x_1 + i x_2 - (y_1 + i y_2)} \varphi(y) \textup{d}  \sigma(y), \quad x \in \Sigma.
\end{equation}
We also consider the formal adjoint $\mathcal{R}^*$ of $\mathcal{R}$ in $L^2(\Sigma)$, which acts on $\varphi \in \mathcal{D}(\Sigma)$ as
\begin{equation} \label{def_Cauchy_transform_dual}
  \mathcal{R}^* \varphi (x) = \frac{2}{\pi} \lim_{\varepsilon \searrow 0} \int_{\Sigma \setminus B(x, \varepsilon)} \frac{\nu_1(x) - i \nu_2(x)}{x_1 - i x_2 - (y_1 - i y_2)} \varphi(y) \textup{d}  \sigma(y), \quad x \in \Sigma.
\end{equation}
Let $C_\Sigma$ be the Cauchy transform on $\Sigma$, i.e. the strongly singular boundary integral operator that acts on $\varphi \in \mathcal{D}(\Sigma)$ evaluated at $x = x_1 + i x_2 \in \Sigma \subset \mathbb{C} \sim \mathbb{R}^2$ as
\begin{equation*}
  \begin{split}
    C_\Sigma \varphi(x) &=  -\frac{1}{i \pi} \lim_{\varepsilon \searrow 0} \int_{\Sigma \setminus B(x, \varepsilon)} \frac{1}{x - \zeta} \varphi(\zeta) \textup{d}  \zeta \\
    &= -\frac{1}{i \pi} \lim_{\varepsilon \searrow 0} \int_{[0, \ell) \setminus (\gamma^{-1}(x) - \varepsilon, \gamma^{-1}(x) + \varepsilon)} \frac{\dot{\gamma}_1(t) + i \dot{\gamma}_2(t)}{x_1 + i x_2 - (\gamma_1(t) + i \gamma_2(t))} \varphi(\gamma(t)) \textup{d} t,
  \end{split}
\end{equation*}
where the first integral is a complex line integral, cf. \cite[Section~5.9]{SV}. 
As we have $\nu(\gamma(s)) = (\dot{\gamma}_2(s), -\dot{\gamma}_1(s))$, $s \in [0, \ell)$, the map $C_\Sigma$ is related to $\mathcal{R}$ by
\begin{equation} \label{relation_R_C}
  \mathcal{R} = 2 C_\Sigma.
\end{equation}
By \cite[Proposition~2.8]{BHOP20} this implies that
\begin{equation} \label{R_PPDO}
  \mathcal{R}, \mathcal{R}^* \in \Psi_\Sigma^0.
\end{equation}
In particular, $\mathcal{R}, \mathcal{R}^*$ give rise to bounded operators in $H^s(\Sigma)$ for all $s \in \mathbb{R}$.
Since $\mathcal{R}$ and $\mathcal{R}^*$ are adjoint to each other, a continuity argument shows for all $\varphi \in H^s(\Sigma)$ and $\psi \in H^{-s}(\Sigma)$, $s \in \mathbb{R}$, that
\begin{equation}
  ( \mathcal{R} \varphi, \psi )_{H^s(\Sigma) \times H^{-s}(\Sigma)} = ( \varphi, \mathcal{R}^* \psi )_{H^s(\Sigma) \times H^{-s}(\Sigma)},
\end{equation}
where $( \cdot, \cdot )_{H^s(\Sigma) \times H^{-s}(\Sigma)}$ denotes the sesquilinear duality product.

\subsection{Analysis of $\mathcal{C}_z$} \label{section_C_z_2d}

In the following proposition, which is the starting point for our further considerations and which is a direct consequence of \cite[Proposition~3.3]{BHOP20}, we provide a link of $\mathcal{R}$ and $\mathcal{R}^*$ defined in~\eqref{def_R_2d} and~\eqref{def_Cauchy_transform_dual}, respectively, and the operator $\mathcal{C}_z$ given in~\eqref{def_C_lambda}. Recall that the map $\Lambda$ is introduced in~\eqref{def_Lambda_2d}.

\begin{proposition} \label{proposition_R_C_2d}
  Let $z \in (-m,m)$. Then there exists an operator $\mathcal{K} \in \Psi^{-2}_\Sigma$ such that 
  \begin{equation} \label{equation_C_z_2d}
    \mathcal{C}_z = \begin{pmatrix}  (z+m) \Lambda^{-2} & -\frac{i}{4} \mathcal{R} (\nu_1 - i \nu_2) \\ \frac{i}{4} (\nu_1 + i \nu_2) \mathcal{R}^* & (z-m) \Lambda^{-2} \end{pmatrix} + \mathcal{K}.
  \end{equation} 
  In particular, for any $r \in \mathbb{R}$ the map $\mathcal{K}$ gives rise to a compact operator $\mathcal{K}: H^r(\Sigma; \mathbb{C}^2) \rightarrow H^{r+1}(\Sigma; \mathbb{C}^2)$ and the realization of $\mathcal{K}$ in $L^2(\Sigma; \mathbb{C}^2)$ is self-adjoint.
\end{proposition}
\begin{proof}
  Denote the operator defined in \cite[equation~(2.7)]{BHOP20} by $\widetilde{\Lambda}$. 
  Then one easily sees that the map $\Lambda$ in~\eqref{def_Lambda_2d} satisfies $\Lambda = \big(\frac{4 \pi}{\ell}\big)^{1/2} \widetilde{\Lambda}$.
  Moreover, the tangential vector $(\dot{\gamma}_1, \dot{\gamma}_2)$ on $\Sigma$ and the normal vector are related by $\nu(\gamma(s)) = (\dot{\gamma}_2(s), -\dot{\gamma}_1(s))$, $s \in [0, \ell]$.  
  Hence, the representation in~\eqref{equation_C_z_2d} follows from~\eqref{relation_R_C} and \cite[Proposition~3.3]{BHOP20}. Moreover, as $\mathcal{K}: H^r(\Sigma; \mathbb{C}^2) \rightarrow H^{r+2}(\Sigma; \mathbb{C}^2)$ is bounded for any $r \in \mathbb{R}$ by  Lemma~\ref{lemma_PPDO_mapping_properties} and $H^{r+2}(\Sigma; \mathbb{C}^2)$ is compactly embedded in $H^{r+1}(\Sigma; \mathbb{C}^2)$ by Rellich's embedding theorem, $\mathcal{K}$ is compact from $H^r(\Sigma; \mathbb{C}^2)$ to $H^{r+1}(\Sigma; \mathbb{C}^2)$. Eventually, since the realization of $\mathcal{C}_z$ in $L^2(\Sigma; \mathbb{C}^2)$ is self-adjoint for $z \in (-m,m)$ by Proposition~\ref{proposition_C_z_basic_properties}~(ii), the map $\mathcal{K}$ must be self-adjoint in $L^2(\Sigma; \mathbb{C}^2)$ as well. Hence, all claims are shown.
\end{proof}

In the following lemma we state a variant of \eqref{equation_C_z_2d} that is particularly useful in the application in the boundary triple framework, as it is done, e.g., in \cite{BHSS22}. Define the matrix $V \in \mathbb{C}^{2 \times 2}$ by
\begin{equation} \label{def_V_2d}
  V = \begin{pmatrix} 1 & 0 \\ 0 & -i (\nu_1-i\nu_2) \end{pmatrix}.
\end{equation}
Then we have the following result:

\begin{lemma} \label{corollary_R_C_bt_2d}
  Let $\mathcal{R}$ and $\mathcal{R}^*$ be defined by~\eqref{def_R_2d} and~\eqref{def_Cauchy_transform_dual}, respectively, let $\mathcal{C}_z$, $z \in (-m,m)$, be given by~\eqref{def_C_lambda}, let $\Lambda$ be as in~\eqref{def_Lambda_2d}, and let $s \in \mathbb{R}$. Then there exists an operator $\mathcal{K}$ that is compact in $H^s(\Sigma; \mathbb{C}^2)$ such that the bounded and everywhere defined operator 
  \begin{equation*}
    4 \Lambda V (\sigma \cdot \nu) \mathcal{C}_z (\sigma \cdot \nu) V^*\Lambda: H^s(\Sigma; \mathbb{C}^2) \rightarrow H^{s-1}(\Sigma; \mathbb{C}^2)
  \end{equation*}
  can be written as
  \begin{equation} 
    4 \Lambda V (\sigma \cdot \nu) \mathcal{C}_z (\sigma \cdot \nu) V^*\Lambda = -\Lambda \begin{pmatrix} 0 & \mathcal{R}^* \\ \mathcal{R} & 0 \end{pmatrix} \Lambda + 4 \begin{pmatrix} z - m & 0 \\ 0 & z+m \end{pmatrix} + \mathcal{K}.
  \end{equation} 
  In particular, the realization of $\mathcal{K}$ in $L^2(\Sigma; \mathbb{C}^2)$ is self-adjoint.
\end{lemma}
\begin{proof}
  The claim follows immediately from Proposition~\ref{proposition_R_C_2d} and a direct calculation taking the identity
  \begin{equation*}
    V (\sigma \cdot \nu) = \begin{pmatrix} 0 & \nu_1 - i \nu_2 \\ -i & 0 \end{pmatrix},
  \end{equation*}
  Lemma~\ref{lemma_PPDO_mapping_properties}~(iii), and~\eqref{equation_mult_op} into account.
\end{proof}

In the following proposition some further properties of $\mathcal{R}$ and $\mathcal{R}^*$ that follow from known results about the Cauchy transform (cf. \cite[Section~4.1.3]{SV}),~\eqref{relation_R_C},  Lemma~\ref{lemma_PPDO_mapping_properties}, and \cite[Proposition~2.8]{BHOP20} are stated.

\begin{proposition} \label{proposition_properties_R_2d}
  Let $\mathcal{R}$ and $\mathcal{R}^*$ be defined by~\eqref{def_R_2d} and~\eqref{def_Cauchy_transform_dual}, respectively. Then, the following holds:
  \begin{itemize}
    \item[(i)] $\mathcal{R} - \mathcal{R}^* \in \Psi^{-\infty}_\Sigma$. In particular, for any $s, t \in \mathbb{R}$ the map $\mathcal{R} - \mathcal{R}^*: H^s(\Sigma) \rightarrow H^t(\Sigma)$ is compact.
    \item[(ii)] One has $\mathcal{R}^2  = 4 I_{2}$ and $(\mathcal{R}^*)^2  = 4 I_{2}$. In particular, $\mathcal{R}  \mathcal{R}^* - 4 I_{2} \in \Psi^{-\infty}_\Sigma$ and for all $s, t \in \mathbb{R}$ the operator $\mathcal{R}  \mathcal{R}^* - 4 I_{2}: H^s(\Sigma) \rightarrow H^t(\Sigma)$ is compact.
    \item[(iii)] $\Lambda^{-2} \mathcal{R} - \mathcal{R}  \Lambda^{-2} \in \Psi^{-2}_\Sigma$.
  \end{itemize}
\end{proposition}

Eventually, we show that $\mathcal{R} + \mathcal{R}^*$ can be written as the difference of two projections, which is useful in the analysis of boundary value and transmission problems for the Dirac equation with critical combinations of the coefficients, see, e.g., \cite{BHSS22} for an application. In order to formulate the result, we define the operators $P_\pm \in \Psi^0_\Sigma$ by
\begin{equation} \label{def_P_pm}
  U P_+ U^{-1} f = \sum_{n=0}^\infty \widehat{f}(n) e_n \quad \text{and} \quad U P_- U^{-1} f = \sum_{n=1}^\infty \widehat{f}(-n) e_{-n}, \qquad f \in \mathcal{D}(\mathbb{T}),
\end{equation}
where $U$ is the map in \eqref{def_U}. Remark that $P_\pm$ are orthogonal projections in $H^s(\Sigma)$ defined by~\eqref{def_H_s} for all $s \in \mathbb{R}$ and with $\mathcal{H}_\pm := \ran P_\pm$ one has $H^s(\Sigma) = \mathcal{H}_+ \oplus \mathcal{H}_-$, $\dim \mathcal{H}_\pm = \infty$, and $\mathcal{H}_\pm \not\subset H^r(\Sigma)$ for any $r > s$. Moreover, $P_\pm$ commute with the operator $\Lambda^t$ defined in~\eqref{def_Lambda_2d}. The map $P_+$ can be interpreted as the projection onto the Hardy space on the unit circle.

\begin{theorem} \label{proposition_anti_commutor_R_2d}
  Let $\mathcal{R}$ and $\mathcal{R}^*$ be defined by~\eqref{def_R_2d} and~\eqref{def_Cauchy_transform_dual}, respectively, and let $P_\pm$ be given as above. Then, there exists $\mathcal{K} \in \Psi^{-\infty}_\Sigma$ such that
    \begin{equation*}
      \mathcal{R} + \mathcal{R}^* = 4 P_+ - 4 P_- + \mathcal{K}.
    \end{equation*}
    In particular, for all $s, t \in \mathbb{R}$ the operator $\mathcal{K}: H^s(\Sigma) \rightarrow H^t(\Sigma)$ is compact.
\end{theorem}
\begin{proof}
  Let $\mathfrak{C} \subset \mathbb{R}^2$ be the circle of radius $1$ centered at the origin and let $U_\mathfrak{C}$ be the map in~\eqref{def_U} defined for $\mathfrak{C}$ instead of $\Sigma$. Then, by \cite[Section~5.8]{SV} one has for the Cauchy transform $C_\mathfrak{C}$ on $\mathfrak{C}$
  \begin{equation*}
    U_\mathfrak{C} C_\mathfrak{C} U_\mathfrak{C}^{-1} = U P_+ U^{-1} - U P_- U^{-1}.
  \end{equation*}
  Moreover, by \cite[Section~5.9 and Theorem~7.6.1]{SV} there exists $\mathcal{K}_1 \in \Psi^{-\infty}$ such that
  \begin{equation*}
    U C_\Sigma U^{-1} = U_\mathfrak{C} C_\mathfrak{C} U_\mathfrak{C}^{-1} + \mathcal{K}_1.
  \end{equation*}
  With~\eqref{relation_R_C} and the last two displayed formulas one gets  
  \begin{equation*}
    U \mathcal{R} U^{-1} = 2 U C_\Sigma U^{-1} = 2 U_\mathfrak{C} C_\mathfrak{C} U_\mathfrak{C}^{-1} + 2 \mathcal{K}_1 = 2 U P_+ U^{-1} - 2 U P_- U^{-1} + 2 \mathcal{K}_1.
  \end{equation*}
  Similarly, there exists $\mathcal{K}_2 \in \Psi^{-\infty}$ such that
  \begin{equation*}
    U \mathcal{R}^* U^{-1} = 2 U P_+ U^{-1} -2 U P_- U^{-1} + \mathcal{K}_2.
  \end{equation*}
  By adding the last two displayed formulas one gets the claim.
\end{proof}

Finally, combining the results from Lemma~\ref{corollary_R_C_bt_2d} and Theorem~\ref{proposition_anti_commutor_R_2d}, we get the following decomposition of $\mathcal{C}_z$ into a positive and a negative part. In particular, this result implies that $\pm \frac{1}{2}$ belongs to the essential spectrum of $\mathcal{C}_z$, when this operator is viewed as a mapping in $H^s(\Sigma; \mathbb{C}^2)$.

\begin{corollary} \label{corollary_positive_negative_2d}
  Let $\mathcal{C}_z$, $z \in (-m,m)$, be defined by~\eqref{def_C_lambda}, let $P_\pm$ be as in~\eqref{def_P_pm}, and let $V$ be given by~\eqref{def_V_2d}. Then, there exists an operator $\mathcal{K} \in \Psi_\Sigma^{-1}$ such that
  \begin{equation*}
    \mathcal{C}_z = \frac{1}{2} (\sigma \cdot \nu) V^* \begin{pmatrix} 0 & P_- - P_+ \\ P_- - P_+ & 0 \end{pmatrix} V (\sigma \cdot \nu) + \mathcal{K}.
  \end{equation*}
  In particular, the realization of $\mathcal{K}$ in $L^2(\Sigma; \mathbb{C}^2)$ is self-adjoint.
\end{corollary}

\section{A more detailed analysis in dimension $q=3$} \label{section_3d}

In this section we are going to prove similar results as in Section~\ref{section_2d}, but instead of pseudodifferential operators we use the theory of pseudo-homogeneous kernels from \cite{N01}. Throughout this section we assume that $\Sigma \subset \mathbb{R}^3$ is the boundary of a bounded and simply connected $C^\infty$-smooth domain. First, in Subsection~\ref{section_pseudohomogeneous_kernels} we follow closely \cite[Section~4.3]{N01} and recall the main definitions and results related to pseudo-homogeneous kernels, which are employed in Subsection~\ref{section_C_z_3d} to analyze $\mathcal{C}_z$.

\subsection{Pseudo-homogeneous kernels} \label{section_pseudohomogeneous_kernels}

With the notion of pseudo-homogeneous kernels one can easily describe mapping properties of the associated integral operators. In the following, we restrict our attention to the three-dimensional case, but a similar construction can be done in any space dimension, cf. \cite[Section~4.3]{N01}. To define pseudo-homogeneous kernels, one first introduces homogeneous kernels \cite[Section~4.3.2]{N01}.

\begin{definition} \label{definition_homogeneous_kernel}
  Let $m \in \mathbb{N}_0$. A function $K \in C^\infty(\mathbb{R}^3 \times (\mathbb{R}^3 \setminus \{ 0 \}))$ is called a homogeneous kernel of class $-m$, if the following conditions hold:
  \begin{itemize}
    \item[(i)] For all $\alpha, \beta \in \mathbb{N}_0^3$ there exists a constant $C_{\alpha, \beta}$ such that
    \begin{equation*}
      \sup_{y \in \mathbb{R}^3} \sup_{|\zeta|=1} \left| \frac{\partial^{|\alpha|}}{\partial y^\alpha} \frac{\partial^{|\beta|}}{\partial \zeta^\beta} K(y,\zeta) \right| \leq C_{\alpha, \beta}.
    \end{equation*}
    \item[(ii)] For any $\beta \in \mathbb{N}_0^3$ with $|\beta|=m$ the function $\frac{\partial^{|\beta|}}{\partial \zeta^\beta} K(y,\zeta)$ is homogeneous of degree $-2$ with respect to $\zeta$.
    \item[(iii)] For any plane $H$ which is constituted by the equation $(h,\zeta)=0$ and any $m$-tuple of vectors $\zeta_1, \dots, \zeta_m$ in $H$ the condition 
    \begin{equation*}
      \int_{S^1} D_\zeta^m K(y,\zeta') (\zeta_1, \dots, \zeta_m) d \zeta'  =0
    \end{equation*}
    holds, where $S^1$ is the intersection of the sphere $S^2$ with $H$.
  \end{itemize}
\end{definition}

In the following lemma we consider a type of homogeneous kernels that will play an important role in the analysis of $\mathcal{C}_z$ in the next section.

\begin{lemma} \label{lemma_homogeneous_kernel}
  Let $\kappa \in C^\infty(\mathbb{R}^3)$ such that $\kappa$ and all of its derivatives are bounded, let $k \in \mathbb{Z}$, and let $\alpha \in \mathbb{N}_0^3$ such that $|\alpha| + 2 k + 1 > 0$. Then the function $K(y,\zeta) = \kappa(y) \zeta^\alpha |\zeta|^{2k-1}$ is a homogeneous kernel of class $-(|\alpha|+2k+1)$.
\end{lemma}
\begin{proof}
  First, one shows as in \cite[Example~4.2]{N01} that $\widetilde{K}(y,\zeta) = \zeta^\alpha |\zeta|^{2k-1}$ is a homogeneous kernel of class $-(|\alpha|+2k+1)$. Since the multiplication by a smooth function depending only on $y$ does not affect the conditions in Definition~\ref{definition_homogeneous_kernel}, the claim on $K$ follows.
\end{proof}

With the help of homogeneous kernels we can introduce now pseudo-homogeneous kernels \cite[Section~4.3.3]{N01}.

\begin{definition} \label{definition_pseudo_homogeneous_kernel}
  Let $m \in \mathbb{N}_0$. A function $K: \mathbb{R}^3 \times (\mathbb{R}^3 \setminus \{ 0 \}) \rightarrow \mathbb{C}$ is called a pseudo-homogeneous kernel of class $-m$, if for any $s \in \mathbb{N}$ there exist $l \in \mathbb{N}$, homogeneous kernels $K_{m+j}$ of class $-(m+j)$, $j \in \{0, \dots, l-1\}$, and a function $K_{m+l}$ that is $s$ times differentiable such that
  \begin{equation*}
    K(y,\zeta) = K_m(y,\zeta) + \sum_{j=1}^{l-1} K_{m+j}(y,\zeta) + K_{m+l}(y,\zeta).
  \end{equation*}
\end{definition}

The important property of pseudo-homogeneous kernels is that one can provide the mapping properties of the associated integral operators \cite[Theorem~4.3.2]{N01}. 

\begin{proposition} \label{proposition_phk_mapping_properties}
  Let $m \in \mathbb{N}_0$ and let $K$ be a pseudo-homogeneous kernel of class $-m$. Then, for any $r \in \mathbb{R}$ the operator that is formally defined by
  \begin{equation*}
    \mathcal{K} \varphi (x) := \int_\Sigma K(y, x-y) \varphi(y) \textup{d} \sigma(y), \quad \varphi \in C^\infty(\Sigma), ~x \in \Sigma,
  \end{equation*}
  gives rise to a bounded operator $\mathcal{K}: H^r(\Sigma) \rightarrow H^{r+m}(\Sigma)$.
\end{proposition}

\subsection{Analysis of $\mathcal{C}_z$} \label{section_C_z_3d}

Let $-\Delta$ be the free Laplacian defined on $H^2(\mathbb{R}^3)$. We use for $\mu \in \rho(-\Delta) = \mathbb{C} \setminus[0,\infty)$ the notation $\mathcal{S}(\mu)$ for the single layer boundary integral operator associated with $-\Delta-\mu$, which acts on a sufficiently smooth function $\varphi: \Sigma \rightarrow  \mathbb{C}$ as
\begin{equation} \label{def_single_layer}
  \mathcal{S}(\mu) \varphi(x) = \int_\Sigma \frac{e^{i \sqrt{\mu}|x - y|}}{4 \pi |x - y|} \varphi(y) \textup{d} \sigma(y), \qquad x \in \Sigma,
\end{equation}
where $\sqrt{\mu}$ is the complex square root satisfying $\textup{Im}\, \sqrt{\mu} > 0$ for $\mu \notin [0, \infty)$, cf. \cite{M00}.
If $\mu < 0$, then it is well-known that $\mathcal{S}(\mu): H^s(\Sigma) \rightarrow H^{s+1}(\Sigma)$ is bounded and bijective for all $s \in \mathbb{R}$ and that the realization of $\mathcal{S}(\mu)$ in $L^2(\Sigma)$ is non-negative and self-adjoint; cf. \cite[Lemma~2.6]{BEHL20} for a similar argument for $\mu=-1$, space dimension $2$, and $s \geq -\frac{1}{2}$; the property for negative $s$ follows then by duality due to the formal symmetry of $\mathcal{S}(\mu)$ for $\mu<0$. 
For a constant $c_\Lambda>0$ we define
\begin{equation} \label{def_Lambda_3d}
  \Lambda := \big( \mathcal{S}(-1)^{-1} + c_\Lambda \big)^{1/2}.
\end{equation}
The following mapping properties $\Lambda$ should be well-known, but for the sake of completeness we present a proof:

\begin{proposition}
  For any $s \in \mathbb{R}$ the operator $\Lambda$ defined in~\eqref{def_Lambda_3d} gives rise to a bounded and bijective map $\Lambda: H^s(\Sigma) \rightarrow H^{s-1/2}(\Sigma)$. Moreover, the realization of $\Lambda$ in $L^2(\Sigma)$ viewed as an unbounded operator defined on $H^{1/2}(\Sigma)$ is self-adjoint.  
\end{proposition}
\begin{proof}
  \textit{Step 1:} First, we show the claim for $s = \frac{1}{2}$ and about the realization of $\Lambda$ in $L^2(\Sigma)$. In the same way as in \cite[Lemma~2.6~(ii)]{BEHL20} one finds that the map $\mathcal{S}(-1)^{1/2}: L^2(\Sigma) \rightarrow H^{1/2}(\Sigma)$ is a non-negative and bijective operator that is self-adjoint in $L^2(\Sigma)$. Hence, for any $c_\Lambda > 0$ the quadratic form
  \begin{equation*}
    \begin{split}
      \mathfrak{a}[\varphi, \psi] &:= \big( \mathcal{S}(-1)^{-1/2} \varphi, \mathcal{S}(-1)^{-1/2} \psi \big)_{L^2(\Sigma)} + c_\Lambda (\varphi, \psi)_{L^2(\Sigma)},\\
      \varphi, \psi \in \dom \mathfrak{a} &= \ran \mathcal{S}(-1)^{1/2} = H^{1/2}(\Sigma),
    \end{split}
  \end{equation*}
  that is associated with the self-adjoint operator $\mathcal{S}(-1)^{-1} + c_\Lambda$, is closed and strictly positive.   
  By the second representation theorem \cite[Chapter~VI, Theorem~2.23]{kato} one has that $ (\mathcal{S}(-1)^{-1} + c_\Lambda)^{1/2}$ is self-adjoint on the set $\dom  (\mathcal{S}(-1)^{-1} + c_\Lambda)^{1/2} = \dom \mathfrak{a} = H^{1/2}(\Sigma)$ and, since this operator is strictly positive, it is bijective from its domain $H^{1/2}(\Sigma)$ into $L^2(\Sigma)$. This implies all claims for $s=\frac{1}{2}$.
  
  \textit{Step 2:} Next, we show the claim for $s = l+\frac{1}{2}$ with $l \in \mathbb{N}$. Since $\mathcal{S}(-1): H^r(\Sigma) \rightarrow H^{r+1}(\Sigma)$ is bijective for all $r \in \mathbb{R}$ and $c_\Lambda > 0$, it is not difficult to see that $\mathcal{S}(-1)^{-1} + c_\Lambda$ gives for any $r \in \mathbb{R}$ rise to a bounded and bijective operator 
  \begin{equation*}
    \mathcal{S}(-1)^{-1} + c_\Lambda: H^{r+1}(\Sigma) \rightarrow H^r(\Sigma).
  \end{equation*}
  In particular, the mapping
  \begin{equation*}
    A_{l,r}:= \big(\mathcal{S}(-1)^{-1} + c_\Lambda \big)^l: H^{l+r}(\Sigma) \rightarrow H^{r}(\Sigma)
  \end{equation*}
  is bounded and bijective for any $r \in \mathbb{R}$. Hence, we conclude with the result from \textit{Step~1} and the spectral theorem that
  \begin{equation*}
    \Lambda := A_{l, 0}^{-1} (\mathcal{S}(-1)^{-1} + c_\Lambda)^{1/2} A_{l, 1/2}: H^{l+1/2}(\Sigma) \rightarrow H^l(\Sigma)
  \end{equation*}
  is a bounded and bijective map that is a restriction of $(\mathcal{S}(-1)^{-1} + c_\Lambda)^{1/2}$.

  \textit{Step~3:} To conclude, we mention first that the statement for $s = -l$ with $l \in \mathbb{N}_0$ follows from the results in \textit{Steps~1} \& \textit{2} by duality and the formal symmetry of $\Lambda$. An interpolation argument yields the claim for arbitrary $s \in \mathbb{R}$. 
\end{proof}

The main object in this section is the operator $\mathcal{R}$ formally acting on a sufficiently smooth function $\varphi: \Sigma \rightarrow \mathbb{C}^2$ as
\begin{equation} \label{def_R_3d}
  \mathcal{R} \varphi (x) = \lim_{\varepsilon \searrow 0} \int_{\Sigma \setminus B(x,\varepsilon)} r(x,y) \varphi(y) \textup{d} \sigma(y),
\end{equation}
where the integral kernel $r$ is the $\mathbb{C}^{2 \times 2}$-valued function
\begin{equation*}
  r(x,y) = - \frac{\sigma \cdot (x-y)}{\pi|x-y|^3} (\sigma \cdot \nu(y)), \qquad x \neq y.
\end{equation*}
The map $\mathcal{R}$ is closely related to the Riesz transform on $\Sigma$. We define also the formal adjoint of $\mathcal{R}$ with respect to the inner product in $L^2(\Sigma; \mathbb{C}^{2})$ by
\begin{equation} \label{def_R_star_3d} 
  \mathcal{R}^* \varphi (x) = \lim_{\varepsilon \searrow 0} \int_{\Sigma \setminus B(x,\varepsilon)} r(y,x)^* \varphi(y) \textup{d} \sigma(y).
\end{equation}
We remark that the definitions of $\mathcal{R}$ and $\mathcal{R}^*$ imply the simple, but useful relations
\begin{equation} \label{commutator_R_R_star}
  (\sigma \cdot \nu) \mathcal{R} = -\mathcal{R}^* (\sigma \cdot \nu) \quad \text{and} \quad (\sigma \cdot \nu) \mathcal{R}^* = -\mathcal{R} (\sigma \cdot \nu).
\end{equation}
In the following proposition we provide a link of $\mathcal{R}$ and $\mathcal{R}^*$ and the operator $\mathcal{C}_z$ given in~\eqref{def_C_lambda}. In particular, this implies the mapping properties of $\mathcal{R}$ and $\mathcal{R}^*$.

\begin{proposition} \label{proposition_R_C_3d}
  Let $\mathcal{R}$ and $\mathcal{R}^*$ be defined by~\eqref{def_R_3d} and~\eqref{def_R_star_3d}, respectively, let $\Lambda$ be as in~\eqref{def_Lambda_3d}, let $\mathcal{C}_z$, $z \in (-m,m)$, be given by~\eqref{def_C_lambda}, and let $s \in \mathbb{R}$. Then $\mathcal{R}$ and $\mathcal{R}^*$ give rise to bounded operators in $H^{s}(\Sigma; \mathbb{C}^2)$ 
  and there exists an operator $\mathcal{K}$ that is compact from $H^s(\Sigma; \mathbb{C}^4)$ to $H^{s+1}(\Sigma; \mathbb{C}^4)$ such that 
  \begin{equation} \label{equation_C_z_3d}
    \mathcal{C}_z = \begin{pmatrix} (z+m) \Lambda^{-2} I_2 & -\frac{i}{4} \mathcal{R} (\sigma \cdot \nu) \\ \frac{i}{4} (\sigma \cdot \nu) \mathcal{R}^* & (z-m) \Lambda^{-2} I_2 \end{pmatrix} + \mathcal{K}.
  \end{equation} 
  In particular, the realization of $\mathcal{K}$ in $L^2(\Sigma; \mathbb{C}^4)$ is self-adjoint.
\end{proposition}


\begin{proof}
  Denote by $\mathcal{T}$ the integral operator that is formally acting on a sufficiently smooth function $\varphi: \Sigma \rightarrow \mathbb{C}^2$ as
  \begin{equation*}
    \mathcal{T} \varphi(x) := \lim_{\varepsilon \searrow 0} \int_{\Sigma \setminus B(x, \varepsilon)}   t(x-y) \varphi(y) \textup{d}\sigma(y), \quad  x \in \Sigma,
  \end{equation*}
  where the integral kernel $t: \mathbb{R}^3 \setminus \{0\} \rightarrow \mathbb{C}^{2 \times 2}$ is given by
  \begin{equation} \label{int_kernel_t}
    t(x) := \left( 1 - i \sqrt{z^2-m^2} | x |\right) \frac{i (\sigma \cdot x)}{ 4 \pi | x |^3}    e^{i \sqrt{z^2-m^2} |x|}.
  \end{equation}
  Recall that $\mathcal{S}(z^2-m^2)$ denotes the single layer boundary integral operator for $-\Delta + m^2 - z^2$. With these notations we conclude from~\eqref{def_G_lambda} that
  \begin{equation} \label{equation_decomposition_C_z}
    \mathcal{C}_z = \begin{pmatrix} (z+m) \mathcal{S}(z^2-m^2) I_2 & \mathcal{T} \\ \mathcal{T} & (z-m) \mathcal{S}(z^2-m^2) I_2 \end{pmatrix}.
  \end{equation}
  Note that \eqref{equation_decomposition_C_z} implies that $\mathcal{T}$ gives rise to a bounded operator in $H^s(\Sigma; \mathbb{C}^2)$, as $\mathcal{C}_z$ in bounded in $H^s(\Sigma; \mathbb{C}^4)$ by Proposition~\ref{proposition_C_z_basic_properties}.
  
  First, we analyze the diagonal terms and show that 
  \begin{equation} \label{equation_single_layer}
    \begin{split}
      \mathcal{S}(z^2-m^2) &= \Lambda^{-2} + \mathcal{K}_1
    \end{split}
  \end{equation}
  with a compact operator $\mathcal{K}_1: H^s(\Sigma) \rightarrow H^{s+1}(\Sigma)$.
  Indeed~\eqref{def_Lambda_3d} implies that
  \begin{equation*} 
    \begin{split}
      \mathcal{S}(z^2-m^2) &= \Lambda^{-2} + \mathcal{K}_2 + \mathcal{K}_3,
    \end{split}
  \end{equation*}
  where 
  \begin{equation*}
    \mathcal{K}_2 := \mathcal{S}(z^2-m^2) - \mathcal{S}(-1) \quad \text{and} \quad \mathcal{K}_3 := \mathcal{S}(-1) - (\mathcal{S}(-1)^{-1} + c_\Lambda)^{-1}.
  \end{equation*}
  With~\eqref{def_single_layer} we get that $\mathcal{K}_2$ is an integral operator with kernel $k_2(x-y)$, where
  \begin{equation*}
    \begin{split}
      k_2(x) &= \frac{e^{-\sqrt{m^2-z^2}|x|}}{4 \pi |x|} - \frac{e^{-|x|}}{4 \pi |x|}, \quad x \in \mathbb{R}^3 \setminus \{ 0 \}.
    \end{split}
  \end{equation*}
  By making a power series expansion we see that the first term in the sums cancels, 
  \begin{equation*}
    \begin{split}
      k_2(x) &= \sum_{k=0}^\infty c_k |x|^{2k} + \sum_{k=0}^\infty d_k |x|^{2k+1}
    \end{split}
  \end{equation*}
  with some suitable coefficients $c_k, d_k \in \mathbb{R}$, and that both sums converge absolutely, as the power series expansion of the exponential function has this property as well. Hence, the first sum is an analytic function, while the second sum is a pseudo-homogeneous kernel of class $-3$ in the sense of Definition~\ref{definition_pseudo_homogeneous_kernel}, cf. Lemma~\ref{lemma_homogeneous_kernel}. Thus, we conclude with Proposition~\ref{proposition_phk_mapping_properties} that 
  $\mathcal{K}_2: H^s(\Sigma) \rightarrow H^{s+3}(\Sigma)$ is bounded for any $s \in \mathbb{R}$ and in particular, $\mathcal{K}_2$ is compact as an operator from $H^s(\Sigma)$ to $H^{s+1}(\Sigma)$.
  
  Next, using the resolvent identity we see that
  \begin{equation} \label{equation_K_3_compact}
    \mathcal{K}_3 = c_\Lambda \mathcal{S}(-1) (\mathcal{S}(-1)^{-1} + c_\Lambda)^{-1} = c_\Lambda \mathcal{S}(-1) \Lambda^{-2},
  \end{equation}
  which is due to the mapping properties of $\mathcal{S}(-1)$ and $\Lambda$ bounded from $H^s(\Sigma)$ to $H^{s+2}(\Sigma)$ and hence compact from $H^s(\Sigma)$ to $H^{s+1}(\Sigma)$. Thus, we have shown~\eqref{equation_single_layer}.

  It remains to consider the anti-diagonal blocks. The integral kernel $t$ of the operator $\mathcal{T}$ in~\eqref{int_kernel_t} has the series expansion
  \begin{equation} \label{power_series}
    t(x) = \frac{i (\sigma \cdot x)}{4 \pi |x|^3} + \sum_{k=0}^\infty \tilde{c}_k (\sigma \cdot x) |x|^{2k} + \sum_{k=0}^\infty \tilde{d}_k (\sigma \cdot x) |x|^{2k-1}
  \end{equation}
  with some suitable coefficients $\tilde{c}_k, \tilde{d}_k \in \mathbb{C}$, as the terms with $|x|^{-2}$ cancel. Since the power series for the exponential function is absolutely converging, both sums in~\eqref{power_series} are absolutely converging as well for any $x \neq 0$. In particular, the first sum defines an analytic function. Moreover, by Lemma~\ref{lemma_homogeneous_kernel} the second sum is a pseudo-homogeneous kernel of class $-2$ in the sense of Definition~\ref{definition_pseudo_homogeneous_kernel}. Hence, we conclude with Proposition~\ref{proposition_phk_mapping_properties} that 
  \begin{equation} \label{equation_R}
    \mathcal{T}  = -\frac{i}{4} \mathcal{R} (\sigma \cdot \nu) + \mathcal{K}_4,
  \end{equation}
  where $\mathcal{K}_4$ is bounded from $H^s(\Sigma; \mathbb{C}^2)$ to $H^{s+2}(\Sigma; \mathbb{C}^2)$ and, in particular, compact from $H^{s}(\Sigma; \mathbb{C}^2)$ to $H^{s+1}(\Sigma; \mathbb{C}^2)$. Combining~\eqref{equation_R} with~\eqref{commutator_R_R_star}, we can also write
  \begin{equation} \label{equation_R_star}
     \mathcal{T} = \frac{i}{4} (\sigma \cdot \nu)  \mathcal{R}^* + \mathcal{K}_4.
  \end{equation}
  In particular, as $\mathcal{T}$ is bounded in $H^s(\Sigma; \mathbb{C}^2)$ and $\sigma \cdot \nu$ is a bijection in $H^s(\Sigma; \mathbb{C}^2)$, we conclude that $\mathcal{R}$ and $\mathcal{R}^*$ give rise to bounded operators in $H^{s}(\Sigma; \mathbb{C}^2)$. Finally, equations \eqref{equation_decomposition_C_z}, \eqref{equation_single_layer}, \eqref{equation_R}, and \eqref{equation_R_star} imply \eqref{equation_C_z_3d}. Furthermore, as the realization of $\mathcal{C}_z$ in $L^2(\Sigma; \mathbb{C}^4)$ is self-adjoint for $z \in (-m,m)$ by Proposition~\ref{proposition_C_z_basic_properties}~(ii), also the operator $\mathcal{K}$ in~\eqref{equation_C_z_3d} must be self-adjoint in $L^2(\Sigma; \mathbb{C}^4)$.
\end{proof}

Let us mention that, since $\mathcal{R}$ and $\mathcal{R}^*$ are formally adjoint to each other with respect to the inner product in $L^2(\Sigma; \mathbb{C}^2)$, a continuity argument shows for all $\varphi \in H^s(\Sigma; \mathbb{C}^2)$ and $\psi \in H^{-s}(\Sigma; \mathbb{C}^2)$, $s \in \mathbb{R}$, that
\begin{equation}
  ( \mathcal{R} \varphi, \psi )_{H^s(\Sigma; \mathbb{C}^2) \times H^{-s}(\Sigma; \mathbb{C}^2)} = ( \varphi, \mathcal{R}^* \psi )_{H^s(\Sigma; \mathbb{C}^2) \times H^{-s}(\Sigma; \mathbb{C}^2)},
\end{equation}
where $( \cdot, \cdot )_{H^s(\Sigma; \mathbb{C}^2) \times H^{-s}(\Sigma; \mathbb{C}^2)}$ denotes the sesquilinear duality product in $H^s(\Sigma; \mathbb{C}^2) \times H^{-s}(\Sigma; \mathbb{C}^2)$.

In the following lemma we state a variant of \eqref{equation_C_z_3d} that is particularly useful in the application in the boundary triple framework, as it is done, e.g., in \cite{BHSS22}. Define the matrix $V \in \mathbb{C}^{4 \times 4}$ by
\begin{equation} \label{def_V_3d}
  V = \begin{pmatrix} I_2 & 0 \\ 0 & -i (\sigma \cdot \nu) \end{pmatrix}.
\end{equation}
Then we have the following result:

\begin{lemma} \label{corollary_R_C_bt}
  Let $\mathcal{R}$ and $\mathcal{R}^*$ be defined by~\eqref{def_R_3d} and~\eqref{def_R_star_3d}, respectively, let $\Lambda$ be as in~\eqref{def_Lambda_3d}, and let $\mathcal{C}_z$, $z \in (-m,m)$, be given by~\eqref{def_C_lambda}. Then, for any $s \in \mathbb{R}$ there exists an operator $\mathcal{K}$ that is compact in $H^s(\Sigma; \mathbb{C}^4)$ such that the bounded and everywhere defined operator 
  \begin{equation*}
    4 \Lambda V (\alpha \cdot \nu) \mathcal{C}_z (\alpha \cdot \nu) V^*\Lambda: H^s(\Sigma; \mathbb{C}^4) \rightarrow H^{s-1}(\Sigma; \mathbb{C}^4)
  \end{equation*}
  can be written as
  \begin{equation} \label{equation_C_z_bt}
    4 \Lambda V (\alpha \cdot \nu) \mathcal{C}_z (\alpha \cdot \nu) V^*\Lambda = -\Lambda \begin{pmatrix} 0 & \mathcal{R}^* \\ \mathcal{R} & 0 \end{pmatrix} \Lambda + 4 \begin{pmatrix} z - m & 0 \\ 0 & z+m \end{pmatrix} + \mathcal{K}.
  \end{equation} 
  In particular, the realization of $\mathcal{K}$ in $L^2(\Sigma; \mathbb{C}^4)$ is self-adjoint.
\end{lemma}
\begin{proof}
  First, equation~\eqref{equation_C_z_3d} and a simple calculation imply 
  \begin{equation} \label{equation_product}
    \begin{split}
      4 \Lambda V (\alpha \cdot \nu) &\mathcal{C}_z (\alpha \cdot \nu) V^*\Lambda  \\
      &= \Lambda \begin{pmatrix} 4 (z-m) (\sigma \cdot \nu) \Lambda^{-2} (\sigma \cdot \nu) & -\mathcal{R}^* \\ -\mathcal{R} & 4 (z+m) \Lambda^{-2} \end{pmatrix} \Lambda + \mathcal{K}_1,
    \end{split}
  \end{equation}
  where $\mathcal{K}_1$ is a compact operator in $H^s(\Sigma; \mathbb{C}^4)$ due to the mapping properties of $\Lambda$.

  By~\eqref{def_Lambda_3d} we can write
  \begin{equation*}
    (\sigma \cdot \nu) \Lambda^{-2} - \Lambda^{-2} (\sigma \cdot \nu) = \mathcal{T} + \mathcal{K}_2,
  \end{equation*}    
   where $\mathcal{T}$ is the integral operator acting as   
  \begin{equation} \label{def_T}
    \mathcal{T} \varphi(x) = \big( (\sigma \cdot \nu) \mathcal{S}(-1) - \mathcal{S}(-1) (\sigma \cdot \nu) \big) \varphi(x) = \int_\Sigma t(x,y) \varphi \textup{d} \sigma(y)
  \end{equation}
  with
  \begin{equation} \label{def_t_int_kernel}
    t(x,y) = \big( \sigma \cdot (\nu(x) - \nu(y)) \big) \frac{e^{-|x-y|}}{ 4 \pi | x -y|}
  \end{equation}
  and $\mathcal{K}_2$ is given by
  \begin{equation*}
    \mathcal{K}_2 = (\sigma \cdot \nu) \big( (\mathcal{S}(-1) + c_\Lambda)^{-1} - \mathcal{S}(-1) \big) + \big( \mathcal{S}(-1) - (\mathcal{S}(-1) + c_\Lambda)^{-1} \big) (\sigma \cdot \nu).
  \end{equation*}
  A similar argument as in~\eqref{equation_K_3_compact} shows that $\mathcal{K}_2$ is compact as a mapping from $H^{s-1/2}(\Sigma; \mathbb{C}^2)$ to $H^{s+1/2}(\Sigma; \mathbb{C}^2)$.

  Next, we show that $\mathcal{T}$ in~\eqref{def_T} is compact as an operator from $H^{s-1/2}(\Sigma; \mathbb{C}^2)$ to $H^{s+1/2}(\Sigma; \mathbb{C}^2)$. Denote by $\tilde{\nu}=(\tilde{\nu}_1, \tilde{\nu}_2, \tilde{\nu}_3)$ a smooth and compactly supported extension of $\nu$ that is defined on $\mathbb{R}^3$. Then, this function has a Taylor series expansion of the form 
  \begin{equation*}
    \tilde{\nu}(x) = \tilde{\nu}(y) + D \tilde{\nu}(y) (x-y) + \sum_{j=1}^3 (x-y)^\top H \tilde{\nu}_j(y) (x-y) e_j + \dots,
  \end{equation*}
  where $D \tilde{\nu}$ and $H \tilde{\nu}_j$ are the Jacobi matrix and the Hessian of $\tilde{\nu}$ and $\tilde{\nu}_j$, respectively, and $e_j$ are the canonical basis vectors in $\mathbb{R}^3$. Using this and the power series expansion of  
  the exponential function we find that the function $t$ defined in~\eqref{def_t_int_kernel} can be written as
  \begin{equation*}
    \begin{split} 
      t(x,y) &= \sigma \cdot \bigg( D \tilde{\nu}(y) (x-y) + \sum_{j=1}^3 (x-y)^\top H \tilde{\nu}_j(y) (x-y) e_j + \dots \bigg) \sum_{k=0}^\infty \frac{|x-y|^{2k-1}}{4 \pi (2k)!} \\
       &\qquad - \sigma \cdot (\tilde{\nu}(x) - \tilde{\nu}(y)) \sum_{k=0}^\infty \frac{|x-y|^{2k}}{4 \pi (2k+1)!}.
    \end{split}
  \end{equation*}
  Since the second term is a smooth function, we find with Lemma~\ref{lemma_homogeneous_kernel} that $t(x,y)$ is a pseudo-homogeneous kernel of class $-2$. Therefore, we get with Proposition~\ref{proposition_phk_mapping_properties} that $\mathcal{T}$ is bounded from $H^{s-1/2}(\Sigma; \mathbb{C}^2)$ to $H^{s+3/2}(\Sigma; \mathbb{C}^2)$ and, in particular, compact from $H^{s-1/2}(\Sigma; \mathbb{C}^2)$ to $H^{s+1/2}(\Sigma; \mathbb{C}^2)$. Hence, we conclude that
  \begin{equation*}
    \Lambda (\sigma \cdot \nu) \Lambda^{-2} (\sigma \cdot \nu) \Lambda = I_2 + \mathcal{K}_3
  \end{equation*}
  holds, where $\mathcal{K}_3$ is compact in $H^s(\Sigma; \mathbb{C}^2)$.
  Together with~\eqref{equation_product} this implies~\eqref{equation_C_z_bt}.
%

  Finally, equation~\eqref{equation_C_z_bt} and the fact that the realization of $\mathcal{C}_z$ in $L^2(\Sigma; \mathbb{C}^4)$ is self-adjoint for $z \in (-m,m)$ by Proposition~\ref{proposition_C_z_basic_properties} yield that also $\mathcal{K}$ must be self-adjoint in $L^2(\Sigma; \mathbb{C}^4)$.
\end{proof}

In the following proposition some further properties of $\mathcal{R}$ and $\mathcal{R}^*$ that follow from known properties of $\mathcal{C}_z$ are stated:

\begin{proposition} \label{proposition_properties_R_3d}
  Let $\mathcal{R}$ and $\mathcal{R}^*$ be defined by~\eqref{def_R_3d} and~\eqref{def_R_star_3d}, respectively, and let $s \in \mathbb{R}$. Then, the following holds:
  \begin{itemize}
    \item[(i)] The operator $\mathcal{R} - \mathcal{R}^*: H^{s-1}(\Sigma; \mathbb{C}^{2}) \rightarrow H^{s}(\Sigma; \mathbb{C}^{2})$ gives rise to a bounded operator.
    \item[(ii)] One has $\mathcal{R}^2  = 4 I_{2}$ and $(\mathcal{R}^*)^2  = 4 I_{2}$. In particular, this implies that 
    \begin{equation} \label{mapping_properties}
      \mathcal{R}  \mathcal{R}^* - 4 I_{2}: H^{s-1}(\Sigma; \mathbb{C}^{2}) \rightarrow H^{s}(\Sigma; \mathbb{C}^{2})
    \end{equation}
    gives rise to a bounded operator.
    \item[(iii)] The operator $\Lambda^{-2} \mathcal{R} - \mathcal{R}  \Lambda^{-2}: H^{s-1}(\Sigma; \mathbb{C}^{2}) \rightarrow H^{s}(\Sigma; \mathbb{C}^{2})$ is compact.
  \end{itemize}
\end{proposition}
\begin{proof}
  (i) Taking the definitions of $\mathcal{R}$ and $\mathcal{R}^*$ into account, we see that their difference is a singular integral operator with kernel
  \begin{equation} \label{int_kernel}
    \begin{split}
      r(x,y) &- r(y,x)^* = -\frac{\sigma \cdot (x-y)}{\pi |x-y|^3} (\sigma \cdot \nu(y))-(\sigma \cdot \nu(x)) \frac{\sigma \cdot (x-y)}{\pi |x-y|^3} \\
      &= \frac{\sigma \cdot (x-y)}{\pi |x-y|^3} \sigma \cdot (\nu(x)-\nu(y))-2 \frac{\nu(x) \cdot (x-y)}{\pi |x-y|^3}
      =: s_1(x,y) + s_2(x,y).
    \end{split}
  \end{equation}
  Following ideas from the proof of \cite[Proposition~2.8]{OV16}, we remark first that the kernel $s_2(x,y) := -2 \frac{\nu(x) \cdot (x-y)}{\pi |x-y|^3}$ is pseudo-homogeneous of class $-1$, as $\nu(x) \cdot (x-y) \sim |x-y|^2$ for $x \rightarrow y$ (for details see the considerations on the kernel $K_1$ in the proof of \cite[Proposition~2.8]{OV16}). 
  Concerning the first term in the sum in~\eqref{int_kernel}, one gets in a similar way as in the study of~\eqref{def_t_int_kernel} with Lemma~\ref{lemma_homogeneous_kernel} that $s_1(x,y)$ is a pseudo-homogeneous kernel of class $-1$ in the sense of Definition~\ref{definition_pseudo_homogeneous_kernel}. 
  Summing up, we find that $r(x,y) - r(y,x)^*$ is a kernel of class $-1$ and hence, by Proposition~\ref{proposition_phk_mapping_properties} the operator $\mathcal{R}-\mathcal{R}^*: H^{s-1}(\Sigma; \mathbb{C}^2) \rightarrow H^{s}(\Sigma; \mathbb{C}^2)$ is bounded, which is the claim of item~(i).
  
  (ii) First, we show that $\mathcal{R}^2  = 4 I_2$ holds. It is known from \cite[Proposition~3.5~(i)]{BEHL18} that $\mathcal{C}_m:=\lim_{z \rightarrow m} \mathcal{C}_z$ exists with respect to the norm in the space of bounded operators in $L^2(\Sigma; \mathbb{C}^4)$ and the limit is given by
  \begin{equation*}
    \mathcal{C}_m = \frac{1}{4} \begin{pmatrix} 2 m \widehat{S} & -i \mathcal{R} (\sigma \cdot \nu) \\ -i \mathcal{R} (\sigma \cdot \nu) & 0 \end{pmatrix}
  \end{equation*}
  with
  \begin{equation*}
    \widehat{S} \varphi(x) = \int_\Sigma \frac{1}{\pi |x-y|} \varphi(y) \textup{d} \sigma(y), \quad \varphi \in L^2(\Sigma; \mathbb{C}^2).
  \end{equation*}
  Combining this with Proposition~\ref{proposition_C_z_basic_properties}~(iii) yields
  \begin{equation*}
    \begin{split}
      I_4 &= 4 \big(\mathcal{C}_m (i \alpha \cdot \nu)\big)^2 
      = 4 \left( \frac{1}{4} \begin{pmatrix} 2 m \widehat{S} & -i \mathcal{R} (\sigma \cdot \nu) \\ -i \mathcal{R} (\sigma \cdot \nu) & 0 \end{pmatrix} \begin{pmatrix} 0 & i (\sigma \cdot \nu) \\ i (\sigma \cdot \nu) & 0 \end{pmatrix} \right)^2 \\
      &= \frac{1}{4} \begin{pmatrix}  \mathcal{R} & 2 i m \widehat{S} (\sigma \cdot \nu) \\ 0& \mathcal{R} \end{pmatrix}^2 = \frac{1}{4} \begin{pmatrix}  \mathcal{R}^2 & 2 i m \mathcal{R} \widehat{S} (\sigma \cdot \nu) + 2 i m \widehat{S} (\sigma \cdot \nu) \mathcal{R} \\ 0 & \mathcal{R}^2 \end{pmatrix},
    \end{split}
  \end{equation*}
  which shows $\mathcal{R}^2=4 I_2$. By taking the adjoint this implies $(\mathcal{R}^*)^2=4 I_2$. Using this we also find that 
  \begin{equation*}
    \mathcal{R}  \mathcal{R}^* - 4 I_{2} = \mathcal{R} \big(\mathcal{R}^* - \mathcal{R} \big) 
  \end{equation*}
  holds.
  This and the result from item~(i) yield the mapping properties in~\eqref{mapping_properties}.
  
  (iii) First we note that Proposition~\ref{proposition_R_C_3d} implies that
  \begin{equation} \label{equation_C_z}
    \mathcal{C}_z (\alpha \cdot \nu) = \frac{1}{4} \begin{pmatrix} -i \mathcal{R} & 4 (z+m) \Lambda^{-2} (\sigma \cdot \nu) \\ 4 (z-m)  \Lambda^{-2} (\sigma \cdot \nu) & i (\sigma \cdot \nu) \mathcal{R}^* (\sigma \cdot \nu) \end{pmatrix} + \mathcal{K}_1
  \end{equation} 
  holds, where $\mathcal{K}_1: H^{s-1}(\Sigma; \mathbb{C}^4) \rightarrow H^{s}(\Sigma; \mathbb{C}^4)$ is compact. Using this representation in $-4 (\mathcal{C}_z (\alpha \cdot \nu))^2 = I_4$, see Proposition~\ref{proposition_C_z_basic_properties}~(iii), we find that the upper right block of this equation has the form 
  \begin{equation*}
    \begin{split}
      0 &= i (z+m) \mathcal{R} \Lambda^{-2} (\sigma \cdot \nu) - i (z+m) \Lambda^{-2} \mathcal{R}^* (\sigma \cdot \nu) + \mathcal{K}_2 \\
      &= i (z+m) \big( \mathcal{R} \Lambda^{-2} - \Lambda^{-2} \mathcal{R} \big) (\sigma \cdot \nu) + i (z+m) \Lambda^{-2} (\mathcal{R} - \mathcal{R}^*) (\sigma \cdot \nu)+ \mathcal{K}_2
    \end{split}
  \end{equation*}
  with a compact operator $\mathcal{K}_2: H^{s-1}(\Sigma; \mathbb{C}^2) \rightarrow H^{s}(\Sigma; \mathbb{C}^2)$. Since $\Lambda^{-2} (\mathcal{R} - \mathcal{R}^*)$ is bounded from $H^{s-1}(\Sigma; \mathbb{C}^2)$ to $H^{s+1}(\Sigma; \mathbb{C}^2)$ by item~(i) and the mapping properties of $\Lambda$, the last equation and Rellich's embedding theorem yield the claim.
\end{proof}

Next, we show that $\mathcal{R} + \mathcal{R}^*$ can be written as the difference of two projections, which is useful in the analysis of boundary value and transmission problems for the Dirac equation with critical combinations of the coefficients, see, e.g., \cite{BHSS22} for an application.

\begin{theorem} \label{proposition_anti_commutor_R_3d}
  Let $\mathcal{R}$ and $\mathcal{R}^*$ be defined as above and let $s \in [-1, 1]$. Then, there exist a bounded operator $\mathcal{K}: H^{s}(\Sigma; \mathbb{C}^{2}) \rightarrow H^{1}(\Sigma; \mathbb{C}^{2})$ and closed subspaces $\mathcal{H}_\pm$ of $H^{s}(\Sigma; \mathbb{C}^{2})$ satisfying $\dim \mathcal{H}_+ = \dim \mathcal{H}_-$ and $H^{s}(\Sigma; \mathbb{C}^{2}) = \mathcal{H}_+ \dot{+} \mathcal{H}_-$ such that the realization of $\mathcal{R} + \mathcal{R}^*$ in the space $H^{s}(\Sigma; \mathbb{C}^{2})$ can be written as
    \begin{equation*}
      \mathcal{R} + \mathcal{R}^* = 4 P_+ - 4 P_- + \mathcal{K},
    \end{equation*}
    where $P_\pm$ is the projection onto $\mathcal{H}_\pm$.
\end{theorem}
\begin{proof}
  The proof of this result is separated in four steps. In \textit{Step 1} we consider the realization of $\mathcal{R} + \mathcal{R}^*$ in the space $L^2(\Sigma; \mathbb{C}^2)$ and prove a spectral representation that is suitable to define the map $\mathcal{K}$. In \textit{Step 2} we show an auxiliary embedding result, which is then used in \textit{Step 3} to conclude the mapping properties of $\mathcal{K}$. Finally, in \textit{Step 4} we show the claims about $\mathcal{H}_\pm$ and $P_\pm$. Throughout the proof, we will denote by $(\cdot, \cdot)$ the inner product in $L^2(\Sigma; \mathbb{C}^2)$.
  
  \textit{Step 1:} Consider $\mathcal{R}$ and $\mathcal{R}^*$ as bounded operators in $L^2(\Sigma; \mathbb{C}^2)$. Then, Proposition~\ref{proposition_properties_R_3d}~(ii) implies that
  \begin{equation} \label{equation_square}
    (\mathcal{R} + \mathcal{R}^*)^2 = \mathcal{R}^2 + (\mathcal{R}^*)^2 + \mathcal{R}^* \mathcal{R} + \mathcal{R} \mathcal{R}^* = 16 I_2 + \mathcal{A}^2,
  \end{equation}
  where we have set
  \begin{equation*}
    \mathcal{A} := i \mathcal{R} - i \mathcal{R}^*.
  \end{equation*}
  By definition, $\mathcal{A}$ is a bounded and self-adjoint operator in $L^2(\Sigma; \mathbb{C}^2)$ and by Proposition~\ref{proposition_properties_R_3d}~(i) $\mathcal{A}$ gives rise to a bounded operator from $L^2(\Sigma; \mathbb{C}^2)$ to $H^1(\Sigma; \mathbb{C}^2)$ and hence, by Rellich's embedding theorem, $\mathcal{A}$ is compact in $L^2(\Sigma; \mathbb{C}^2)$. Thus, the spectral theorem implies that
  \begin{equation} \label{spectrum_inclusion}
    \sigma(\mathcal{R} + \mathcal{R}^*) \subset \left\{ \pm \sqrt{16 + \lambda^2}: \lambda \in \sigma(\mathcal{A}) \right\}
  \end{equation}
  and $\sigma(\mathcal{R} + \mathcal{R}^*) \setminus \{ \pm 4 \}$ is purely discrete. We claim that the inclusion in~\eqref{spectrum_inclusion} is actually an equality. To see this, it suffices to show that $\sigma_\textup{p}(\mathcal{R} + \mathcal{R}^*)$ is symmetric with respect to the origin. Let $\mu \in \sigma_\textup{p}(\mathcal{R} + \mathcal{R}^*)$ and $e \in L^2(\Sigma; \mathbb{C}^2)$ such that $(\mathcal{R} + \mathcal{R}^*) e = \mu e$. With~\eqref{commutator_R_R_star} we see that
  \begin{equation} \label{equation_isomorphism}
    (\mathcal{R} + \mathcal{R}^*) (\sigma \cdot \nu) e = -(\sigma \cdot \nu) (\mathcal{R} + \mathcal{R}^*)  e = -\mu (\sigma \cdot \nu) e,
  \end{equation}
  i.e. also $-\mu \in \sigma_\textup{p}(\mathcal{R} + \mathcal{R}^*)$ and hence
  \begin{equation*} 
    \sigma(\mathcal{R} + \mathcal{R}^*) = \left\{ \pm \sqrt{16 + \lambda^2}: \lambda \in \sigma(\mathcal{A}) \right\}.
  \end{equation*}
  Thus, if we denote by $\mu_k = \text{sign} \, \mu_k \cdot \sqrt{16+\lambda_k^2}$ all eigenvalues of $\mathcal{R} + \mathcal{R}^*$ with some $\lambda_k \in \sigma_\textup{p}(\mathcal{A})$ and by $e_k^\pm$ the corresponding eigenfunctions (where the superscript corresponds to the sign of the associated eigenvalue and which are by~\eqref{equation_square} also eigenfunctions of $\mathcal{A}$), and that are by the spectral theorem an orthonormal basis in $L^2(\Sigma; \mathbb{C}^2)$, we can write
  \begin{equation} \label{decomposition_R_R_star}
    \begin{split}
      (\mathcal{R} + \mathcal{R}^*) \varphi = \sum_{k} \mu_k (\varphi,e_k^\pm)e_k^\pm = 4 \widetilde{P}_+ \varphi -4 \widetilde{P}_- \varphi + \widetilde{\mathcal{K}} \varphi
    \end{split}
  \end{equation}
  with 
  \begin{equation} \label{def_K}
    \begin{split}
      \widetilde{P}_+ \varphi &:= \sum_{\mu_k>0} (\varphi, e_k^+) e_k^+, \qquad \widetilde{P}_- \varphi := \sum_{\mu_k<0} (\varphi, e_k^-) e_k^-, \\
      \widetilde{\mathcal{K}} \varphi &:= \sum_{\mu_k>0} \Big( \sqrt{16 + \lambda_k^2} - 4 \Big) (\varphi,e_k^+) e_k^+ +  \sum_{\mu_k<0} \Big( 4-\sqrt{16 + \lambda_k^2} \Big) (\varphi,e_k^-) e_k^-.
    \end{split}
  \end{equation}
  It remains to show that $\widetilde{P}_\pm$ give rise to bounded operators $P_\pm$ in $H^{s}(\Sigma; \mathbb{C}^2)$, that $\mathcal{H}_\pm := \ran P_\pm$ fulfil $H^{s}(\Sigma; \mathbb{C}^2) = \mathcal{H}_+ \dot{+} \mathcal{H}_-$, and that $\widetilde{\mathcal{K}}$ can be extended to a bounded operator from $H^{s}(\Sigma; \mathbb{C}^2)$ to $H^{1}(\Sigma; \mathbb{C}^2)$.
  
  \textit{Step 2:} Define the space $\mathcal{G} := \ran \mathcal{A}$ and endow it with the norm
  \begin{equation*}
    \begin{split}
      \| \varphi \|_\mathcal{G}^2 &= \big\| (\mathcal{A} \upharpoonright (L^2(\Sigma; \mathbb{C}^2) \ominus \ker \mathcal{A}))^{-1} \varphi \big\|_{L^2(\Sigma; \mathbb{C}^2)}^2 + \| \varphi \|_{L^2(\Sigma; \mathbb{C}^2)}^2 \\
      &=\sum_{\lambda_k \in \sigma(\mathcal{A}) \setminus \{ 0 \}} (\lambda_k^{-2} + 1) |(\varphi,e_k^\pm)|^2. 
    \end{split}  
  \end{equation*}
  Remark that $(\mathcal{A} \upharpoonright (L^2(\Sigma; \mathbb{C}^2) \ominus \ker \mathcal{A}))^{-1}$ is self-adjoint in the Hilbert space $\overline{\ran \mathcal{A}}$ and hence, it is closed and $\mathcal{G}$ is a Hilbert space. Clearly, those eigenfunctions $e_k^\pm$ of $\mathcal{A}$ that correspond to eigenvalues $\lambda_k \neq 0$ are an orthogonal basis   in $\mathcal{G}$.
  Note that Proposition~\ref{proposition_properties_R_3d}~(i) implies that $\mathcal{G} = \ran (\mathcal{R} - \mathcal{R}^*) \subset H^1(\Sigma; \mathbb{C}^2)$. In this step, we show that the embedding $\iota: \mathcal{G} \rightarrow H^1(\Sigma; \mathbb{C}^2)$ is continuous. 
  
  For that, we show that $\iota$ is closed. So let $(\varphi_n) \subset \mathcal{G}$ such that $\varphi_n \rightarrow \varphi$ in $\mathcal{G}$ and $\iota \varphi_n = \varphi_n \rightarrow \psi$ in $H^1(\Sigma; \mathbb{C}^2)$. Then, as $\mathcal{G}$ is complete $\varphi \in \mathcal{G}$ and by the same argument $\psi \in H^1(\Sigma; \mathbb{C}^2)$. Since both $\mathcal{G}$ and $H^1(\Sigma; \mathbb{C}^2)$ are boundedly embedded in $L^2(\Sigma; \mathbb{C}^2)$, we have that
  $(\varphi_n)$ converges to $\varphi$ and $\psi$ in $L^2(\Sigma; \mathbb{C}^2)$. This can only be true if $\varphi=\psi$. This finishes the proof that $\iota$ is closed and thus, continuous.
  
  \textit{Step 3:} In this step we prove that the map $\widetilde{\mathcal{K}}$ introduced in~\eqref{def_K} can be extended to a bounded operator from $H^{-1}(\Sigma; \mathbb{C}^2)$ to $H^1(\Sigma; \mathbb{C}^2)$, which gives then rise to a bounded mapping $\mathcal{K}: H^{s}(\Sigma; \mathbb{C}^2)  \rightarrow H^{1}(\Sigma; \mathbb{C}^2)$.
  
  To show the stated boundedness property, let $\varphi \in L^2(\Sigma; \mathbb{C}^2)$ be fixed. We claim that $\widetilde{\mathcal{K}} \varphi \in \mathcal{G}$. To see this, we use that the definition of $\widetilde{\mathcal{K}}$ in~\eqref{def_K} yields $|(\widetilde{\mathcal{K}} \varphi,e_k^\pm)| = (\sqrt{16+\lambda_k^2}-4)|(\varphi, e_k^\pm)|$ and thus,
  \begin{equation} \label{boundedness_K}
    \begin{split}
      \| \widetilde{\mathcal{K}} \varphi \|_{\mathcal{G}}^2
      &= \sum_{\lambda_k \in \sigma(\mathcal{A}) \setminus \{ 0 \}} (\lambda_k^{-2} + 1) |(\widetilde{\mathcal{K}} \varphi,e_k^\pm)|^2 \\
      &= \sum_{\lambda_k \in \sigma(\mathcal{A}) \setminus \{ 0 \}} (\lambda_k^{-2} + 1) \left(\frac{\lambda_k^2}{\sqrt{16+\lambda_k^2}+4}\right)^2 |(\varphi,e_k^\pm)|^2 \\
      &\leq c_1 \sum_{\lambda_k \in \sigma(\mathcal{A}) \setminus \{ 0 \}} \lambda_k^2 |(\varphi,e_k^\pm)|^2 = c_1\| \mathcal{A} \varphi \|_{L^2(\Sigma; \mathbb{C}^2)}^2 < \infty.
    \end{split}
  \end{equation}
  Therefore, we have shown $\widetilde{\mathcal{K}} \varphi \in \mathcal{G}$.  
  Combining now~\eqref{boundedness_K} with Proposition~\ref{proposition_properties_R_3d}~(i) (applied with $s=0$) and the result in \textit{Step~2} we conclude
  \begin{equation*}
    \begin{split}
      \| \widetilde{\mathcal{K}} \varphi \|_{H^1(\Sigma; \mathbb{C}^2)}^2 &\leq c_2 \| \widetilde{\mathcal{K}} \varphi \|_{\mathcal{G}}^2
      \leq c_3\| \mathcal{A} \varphi \|_{L^2(\Sigma; \mathbb{C}^2)}^2 \leq c_4 \| \varphi \|_{H^{-1}(\Sigma; \mathbb{C}^2)}. 
    \end{split}
  \end{equation*}
  Since $L^2(\Sigma; \mathbb{C}^2)$ is a dense subspace of $H^{-1}(\Sigma; \mathbb{C}^2)$, the latter estimate shows that $\widetilde{\mathcal{K}}$ can indeed be extended to a bounded map from $H^{-1}(\Sigma; \mathbb{C}^2)$ to $H^1(\Sigma; \mathbb{C}^2)$.
  
  \textit{Step 4:} It remains to show that $\widetilde{P}_\pm$ give rise to bounded operators $P_\pm$ in $H^{s}(\Sigma; \mathbb{C}^2)$ and that $\mathcal{H}_\pm := \ran P_\pm$ fulfil $H^{s}(\Sigma; \mathbb{C}^2) = \mathcal{H}_+ \dot{+} \mathcal{H}_-$. To see the first claim, remark that the eigenfunctions $e_k^\pm$ of $\mathcal{R} + \mathcal{R}^*$ form an orthonormal basis of $L^2(\Sigma; \mathbb{C}^2)$ and that the definition of $\widetilde{P}_\pm$ in~\eqref{def_K} implies $\widetilde{P}_+ + \widetilde{P}_- = I_2$. Hence, formula~\eqref{decomposition_R_R_star} yields
  \begin{equation} \label{relation_P_pm_R}
    \widetilde{P}_\pm = \pm \frac{1}{8} \big( \mathcal{R} + \mathcal{R}^* \pm 4 I_2 - \widetilde{\mathcal{K}} \big).
  \end{equation}
  Since the operator on the right hand side gives rise to a continuous map in $H^{s}(\Sigma; \mathbb{C}^2)$, the same is true for the left hand side and we denote this realization by $P_\pm$. 
  
  It remains to show that the direct sum decomposition $H^s(\Sigma; \mathbb{C}^2) = \mathcal{H}_+ \dot{+} \mathcal{H}_-$ holds. First, we note that~\eqref{relation_P_pm_R} implies $P_+ + P_- = I_2$ and hence, $H^s(\Sigma; \mathbb{C}^2) = \mathcal{H}_+ + \mathcal{H}_-$. To see that the sum is direct, assume that $\varphi \in \mathcal{H}_+ \cap \mathcal{H}_-$. Then, there exist $\varphi_-, \varphi_+ \in H^s(\Sigma; \mathbb{C}^2)$ such that
  \begin{equation} \label{equation_intersection}
    \varphi = P_+ \varphi_+ = P_- \varphi_-.
  \end{equation}
  Note that the definition of $\widetilde{P}_\pm$ in~\eqref{def_K} implies that $\widetilde{P}_\pm \widetilde{P}_\mp = 0$. Thus, also $P_\pm P_\mp = 0$. An Application of $P_\pm$ to~\eqref{equation_intersection} shows then $P_\pm \varphi = 0$, i.e. $\varphi = P_+ \varphi + P_- \varphi = 0$. Therefore, $H^s(\Sigma; \mathbb{C}^2) = \mathcal{H}_+ \dot{+} \mathcal{H}_-$ as a direct sum and all claims are shown.
\end{proof}

\begin{remark} \label{remark_isomorphism}
  It follows from the proof of Theorem~\ref{proposition_anti_commutor_R_3d} that $\sigma \cdot \nu$ is an isomorphism from $\mathcal{H}_\pm$ to $\mathcal{H}_\mp$. Indeed, it is shown in~\eqref{equation_isomorphism} that $(\sigma \cdot \nu) \ker (\mathcal{R} + \mathcal{R}^* - \mu) = \ker (\mathcal{R} + \mathcal{R}^* + \mu)$ holds for any $\mu \neq 0$. Since $\mathcal{H}_\pm$ is defined for $s \in [-1,0]$ as the completion of 
  \begin{equation*}
    \begin{split}
      \ran \widetilde{P}_\pm = \textup{span} \big\{ u \in L^2(\Sigma; \mathbb{C}^2): u \in \ker (\mathcal{R} + \mathcal{R}^* - \mu) \text{ for } \pm \mu > 0 \big\}
    \end{split}
  \end{equation*}
  in $H^{s}(\Sigma; \mathbb{C}^2)$ and for $s \in [0,1]$ as the completion of
  \begin{equation*}
    \begin{split} 
      \textup{span} \big\{ u \in L^2(\Sigma; \mathbb{C}^2): u \in \ker (\mathcal{R} + \mathcal{R}^* - \mu) \text{ for } \pm \mu > 4 \big\} &\\
      + \big(\ker(\mathcal{R} + \mathcal{R}^* \mp 4) \cap H^s(\Sigma; \mathbb{C}^2) &\big)
    \end{split}
  \end{equation*}
  in $H^{s}(\Sigma; \mathbb{C}^2)$,
  cf.~\eqref{def_K}, and $\sigma \cdot \nu$ gives rise to a bounded map in $H^{s}(\Sigma; \mathbb{C}^2)$, this map is indeed an isomorphism between $\mathcal{H}_\pm$ and $\mathcal{H}_\mp$.
\end{remark}

Eventually, we state a result about a special operator in $L^2(\Sigma; \mathbb{C}^2)$ involving the realizations $\widetilde{\mathcal{R}}, \widetilde{\mathcal{R}}^*$ of $\mathcal{R}, \mathcal{R}^*$ in $H^{-1/2}(\Sigma; \mathbb{C}^2)$. In the following proposition we show that zero belongs to the essential spectrum of $\Lambda (\widetilde{\mathcal{R}} \widetilde{\mathcal{R}}^* - 4 I_2) \Lambda$ under some suitable assumptions on the surface $\Sigma$. Remark that in the recent paper \cite{BP22}, that was developed independently from this note, a related result was shown with pseudodifferential techniques.

\begin{proposition}
  Assume that there exists an open subset $\Sigma_0 \subset \Sigma$ such that $\Sigma_0$ is contained in a plane. Then, there exists a sequence $(\varphi_n) \subset L^2(\Sigma; \mathbb{C}^2)$ such that $\| \varphi_n \|_{L^2(\Sigma; \mathbb{C}^2)} = 1$, $\varphi_n$ converges weakly to zero, and 
  $\Lambda (\widetilde{\mathcal{R}} \widetilde{\mathcal{R}}^* - 4 I_2) \Lambda \varphi_n \rightarrow 0$ in $L^2(\Sigma; \mathbb{C}^2)$. It is possible to choose this sequence such that $(\Lambda \varphi_n) \subset \mathcal{H}_\pm$.
\end{proposition}
\begin{proof}
  The proof of this proposition is split into two steps. In \textit{Step~1} we follow closely arguments from \cite[Theorem~5.9~(i)]{BH20} to show that there exists a singular sequence $(\varphi_n)$ for $\Lambda (\widetilde{\mathcal{R}} \widetilde{\mathcal{R}}^* - 4 I_2) \Lambda$, while in \textit{Step~2} we prove that this sequence $(\varphi_n)$ can be chosen such that $\Lambda \varphi_n \in \mathcal{H}_\pm$.

  \textit{Step~1:} We show the existence of the singular sequence for $\Lambda (\widetilde{\mathcal{R}} \widetilde{\mathcal{R}}^* - 4 I_2) \Lambda$. For this we note that by Proposition~\ref{proposition_properties_R_3d}~(ii) the operator $\Lambda (\widetilde{\mathcal{R}} \widetilde{\mathcal{R}}^* - 4 I_2) \Lambda$ is bounded and self-adjoint in $L^2(\Sigma; \mathbb{C}^2)$  and hence, the claim follows, if we can show that $0 \in \sigma_\textup{ess}(\Lambda (\widetilde{\mathcal{R}} \widetilde{\mathcal{R}}^* - 4 I_2) \Lambda)$, which is true, if 
  \begin{equation} \label{condition_0_sigma_ess}
    \nexists \mathcal{G} \subset L^2(\Sigma; \mathbb{C}^2): \dim \mathcal{G} < \infty \text{ and } \ran (\Lambda (\widetilde{\mathcal{R}} \widetilde{\mathcal{R}}^* - 4 I_2) \Lambda) = L^2(\Sigma; \mathbb{C}^2) \ominus \mathcal{G}.
  \end{equation}
  Let $\Sigma_1 \subset \Sigma_0$ such that $\overline{\Sigma_1} \subset \Sigma_0$, where the closure is understood in the relative topology on $\Sigma_0$. Following closely arguments from Step 4 from the proof of \cite[Theorem~5.9~(i)]{BH20}, we will show that $(( \widetilde{\mathcal{R}}^* - \widetilde{\mathcal{R}}) \Lambda \varphi) \upharpoonright \Sigma_1 \in H^1(\Sigma_1; \mathbb{C}^2)$ for all $\varphi \in L^2(\Sigma; \mathbb{C}^2)$. Since Proposition~\ref{proposition_properties_R_3d}~(ii) implies that $\Lambda \widetilde{\mathcal{R}}$ is a bijection from $H^{1/2}(\Sigma; \mathbb{C}^2)$ to $L^2(\Sigma; \mathbb{C}^2)$, this yields~\eqref{condition_0_sigma_ess}.
  
  Recall that $\widetilde{\mathcal{R}}^* - \widetilde{\mathcal{R}}$ is an integral operator with integral kernel 
  \begin{equation*}
    t(x,y) := r(y,x)^* - r(x,y) = \frac{\sigma \cdot (x-y)}{\pi |x-y|^3} \sigma \cdot (\nu(y)-\nu(x))+2 \frac{\nu(x) \cdot (x-y)}{\pi |x-y|^3},
  \end{equation*}
  cf.~\eqref{int_kernel}. Clearly, $t(x,y)=0$ holds for all $x,y \in \Sigma_0$. Let $U_1 \subset \mathbb{R}^2$ and $\phi: U_1 \rightarrow \mathbb{R}^3$ be a map that parametrizes $\Sigma_1$, i.e. $\ran \phi = \Sigma_1$. Since for any $y \in \Sigma$ the map $U_1 \ni u \mapsto t(\phi(u), y)$ is $C^\infty$-smooth and $\Sigma \ni y \mapsto \partial_{u_j} t(\phi(u),y)$ is $C^1$-smooth for any fixed $j \in \{ 1,2 \}$ and $u \in U_1$, we find for any $\psi \in L^2(\Sigma; \mathbb{C}^2)$ that
  \begin{equation*}
    \partial_{u_j} \big( (\widetilde{\mathcal{R}}^* - \widetilde{\mathcal{R}}) \psi \big)(\phi(u)) = \int_\Sigma \partial_{u_j} t(\phi(u), y) \psi(y) \text{d} \sigma(y)
  \end{equation*}
  holds. With this, we get
  \begin{equation*}
    \begin{split}
      \big\| \partial_{u_j} \big( (\widetilde{\mathcal{R}}^* - \widetilde{\mathcal{R}}) \psi \big)  \big\|_{L^2(\Sigma_1; \mathbb{C}^2)}^2
      &= C_1 \int_{U_1} \left| \int_\Sigma \partial_{u_j} t(\phi(u), y) \psi(y) \text{d} \sigma(y) \right|^2 \text{d} u \\
      &= C_1 \int_{U_1} \left| \big( \partial_{u_j} t(\phi(u), \cdot),  \psi \big)_{H^{1/2}(\Sigma; \mathbb{C}^2) \times H^{-1/2}(\Sigma; \mathbb{C}^2)} \right|^2 \text{d} u \\
      & \leq C_1 \int_{U_1} \big\| \partial_{u_j} t(\phi(u), \cdot) \big\|_{H^{1/2}(\Sigma; \mathbb{C}^{2 \times 2})}^2 \cdot \| \psi \|_{ H^{-1/2}(\Sigma; \mathbb{C}^2)}^2 \text{d} u \\
      &=C_2 \| \psi \|_{ H^{-1/2}(\Sigma; \mathbb{C}^2)}^2.
    \end{split}
  \end{equation*}
  By continuity, this can be extended for all $\psi \in H^{-1/2}(\Sigma; \mathbb{C}^2)$, which shows that $( \widetilde{\mathcal{R}}^* - \widetilde{\mathcal{R}}) \Lambda \varphi) \upharpoonright \Sigma_1 \in H^1(\Sigma_1; \mathbb{C}^2)$ for all $\varphi \in L^2(\Sigma; \mathbb{C}^2)$ and yields the claim of the first step.

  \textit{Step~2:} By the result of \textit{Step~1} there exists a sequence $(\varphi_n) \subset L^2(\Sigma; \mathbb{C}^2)$ such that $\| \varphi_n \|_{L^2(\Sigma; \mathbb{C}^2)} = 1$, $\varphi_n$ converges weakly to zero, and 
  $\Lambda (\widetilde{\mathcal{R}} \widetilde{\mathcal{R}}^* - 4 I_2) \Lambda \varphi_n \rightarrow 0$ in $L^2(\Sigma; \mathbb{C}^2)$. It remains to verify that this sequence can be chosen such that $\Lambda \varphi_n \in \mathcal{H}_\pm$. 
  
  First, we show that at least one of the sequences $(\Lambda^{-1} P_\pm \Lambda \varphi_n)$ has the claimed properties. Since by Proposition~\ref{proposition_properties_R_3d}~(ii) the map $\mathcal{R}$ is bijective in $H^{1/2}(\Sigma; \mathbb{C}^2)$ with $\mathcal{R}^2 = 4 I_2$, we find with the mapping properties of $\Lambda$ that 
  \begin{equation} \label{equation_convergence}
    (\widetilde{\mathcal{R}}^* - \widetilde{\mathcal{R}}) \Lambda \varphi_n = \frac{1}{4} \widetilde{\mathcal{R}} ( \widetilde{\mathcal{R}} \widetilde{\mathcal{R}}^* - 4 I_2) \Lambda \varphi_n \rightarrow 0 \quad \text{in} \quad H^{1/2}(\Sigma; \mathbb{C}^2).
  \end{equation}
  Taking the construction of $P_\pm$ in~\eqref{relation_P_pm_R} into account, we conclude that
  \begin{equation*}
    \begin{split}
      (\widetilde{\mathcal{R}}^* - \widetilde{\mathcal{R}}) P_\pm \Lambda \varphi_n &= \pm \frac{1}{8} (\widetilde{\mathcal{R}}^* - \widetilde{\mathcal{R}}) \big( \widetilde{\mathcal{R}} + \widetilde{\mathcal{R}}^* \pm 4 I_2 - \mathcal{K} \big) \Lambda \varphi_n \\
      &= \mp \frac{1}{8} \big( \mathcal{R} + \mathcal{R}^* \mp 4 I_2 \big) (\widetilde{\mathcal{R}}^* - \widetilde{\mathcal{R}})  \Lambda \varphi_n \mp \frac{1}{8} (\widetilde{\mathcal{R}}^* - \widetilde{\mathcal{R}}) \mathcal{K} \Lambda \varphi_n \rightarrow 0
    \end{split}
  \end{equation*}
  in $H^{1/2}(\Sigma; \mathbb{C}^2)$ due to~\eqref{equation_convergence}, as $\mathcal{R}, \mathcal{R}^*$ give rise to bounded operators in $H^{1/2}(\Sigma; \mathbb{C}^2)$ by Proposition~\ref{proposition_R_C_3d} and $\mathcal{K}: H^{-1/2}(\Sigma; \mathbb{C}^2) \rightarrow H^{1/2}(\Sigma; \mathbb{C}^2)$ is compact by Theorem~\ref{proposition_anti_commutor_R_3d} and hence, it turns the weakly convergent sequence $(\Lambda \varphi_n)$ into a strongly convergent one. Using again $\mathcal{R}^2 = 4 I_2$ and the mapping properties of $\Lambda$, we conclude that $\Lambda (\widetilde{\mathcal{R}} \widetilde{\mathcal{R}}^* - 4 I_2) P_\pm \Lambda \varphi_n \rightarrow 0$ in $L^2(\Sigma; \mathbb{C}^2)$. Moreover, both sequences $(\Lambda^{-1} P_\pm \Lambda \varphi_n)$ are weakly converging to zero in $L^2(\Sigma; \mathbb{C}^2)$ and at least one of these sequences satisfies $\| \Lambda^{-1} P_\pm \Lambda \varphi_n \|_{L^2(\Sigma; \mathbb{C}^2)} \geq c$ for a constant $c > 0$. Hence, at least one of the sequences $(\Lambda^{-1} P_\pm \Lambda \varphi_n) \subset L^2(\Sigma; \mathbb{C}^2)$ is a singular sequence for $\Lambda (\widetilde{\mathcal{R}} \widetilde{\mathcal{R}}^* - 4 I_2) \Lambda$.

  So we can assume w.l.o.g. that $(\Lambda \varphi_n) \subset \mathcal{H}_+$ such that $\| \varphi_n \|_{L^2(\Sigma; \mathbb{C}^2)} = 1$, $\varphi_n$ converges weakly to zero, and 
  $\Lambda (\widetilde{\mathcal{R}} \widetilde{\mathcal{R}}^* - 4 I_2) \Lambda \varphi_n \rightarrow 0$ in $L^2(\Sigma; \mathbb{C}^2)$. But then, in view of Remark~\ref{remark_isomorphism}, also the sequence defined by $\psi_n := \Lambda^{-1} (\sigma \cdot \nu) \Lambda \varphi_n$ satisfies $\| \psi_n \|_{L^2(\Sigma; \mathbb{C}^2)} \geq c$ for some $c>0$, $\psi_n$ converges weakly to zero, and $\Lambda \psi_n \in \mathcal{H}_-$. Moreover, the relations in~\eqref{commutator_R_R_star} and Proposition~\ref{proposition_properties_R_3d}~(ii) imply that  
  \begin{equation*}
    \begin{split}
      \Lambda (\widetilde{\mathcal{R}} \widetilde{\mathcal{R}}^* - 4 I_2) \Lambda \psi_n &= \Lambda (\widetilde{\mathcal{R}} \widetilde{\mathcal{R}}^* - 4 I_2) (\sigma \cdot \nu) \Lambda \varphi_n = \Lambda (\sigma \cdot \nu) (\widetilde{\mathcal{R}}^* \widetilde{\mathcal{R}} - 4 I_2)  \Lambda \varphi_n \\
      &= \Lambda (\sigma \cdot \nu) \widetilde{\mathcal{R}}^* \widetilde{\mathcal{R}}\left(I_2 - \frac{1}{4} \widetilde{\mathcal{R}} \widetilde{\mathcal{R}}^* \right)  \Lambda \varphi_n \rightarrow 0
    \end{split}
  \end{equation*}
  in $L^2(\Sigma; \mathbb{C}^2)$. Thus, also $(\psi_n)$ fulfils all claimed properties and all claims are shown.
\end{proof}

Finally, combining the results from Lemma~\ref{corollary_R_C_bt} and Theorem~\ref{proposition_anti_commutor_R_3d}, we get the following decomposition of $\mathcal{C}_z$ into a positive and a negative part. In particular, this result implies that $\pm \frac{1}{2}$ belongs to the essential spectrum of $\mathcal{C}_z$, when this operator is viewed as a mapping in $H^s(\Sigma; \mathbb{C}^4)$, $s \in [-1, 1]$.

\begin{corollary} \label{corollary_positive_negative_3d}
  Let $\mathcal{C}_z$, $z \in (-m,m)$, be defined by~\eqref{def_C_lambda}, let $P_\pm$ be as in~Theorem~\ref{proposition_anti_commutor_R_3d}, and let $V$ be given by~\eqref{def_V_3d}. Then, for any $s \in [-1, 1]$ there exists a bounded operator $\mathcal{K}: H^{s}(\Sigma; \mathbb{C}^4) \rightarrow H^\tau(\Sigma; \mathbb{C}^4)$ with $\tau = \min\{1, s+1\}$ such that
  \begin{equation*}
    \mathcal{C}_z = \frac{1}{2} (\alpha \cdot \nu) V^* \begin{pmatrix} 0 & P_- - P_+ \\ P_- - P_+ & 0 \end{pmatrix} V (\alpha \cdot \nu) + \mathcal{K}.
  \end{equation*}
  In particular, the realization of $\mathcal{K}$ in $L^2(\Sigma; \mathbb{C}^4)$ is self-adjoint.
\end{corollary}

\begin{remark} \label{remark_Hardy_spaces}
  In \cite[Section~4.2]{AMSV22} projections that are defined by
  \begin{equation*}
    Q_\pm := \frac{1}{2} I_2 \pm \frac{1}{4} \mathcal{R} \quad \text{and} \quad Q_\pm^* = \frac{1}{2} I_2 \pm \frac{1}{4} \mathcal{R}^*
  \end{equation*}
  were considered and they were identified as skew projections onto Hardy spaces. Note that Theorem~\ref{proposition_anti_commutor_R_3d} and $P_+ + P_- = I_2$ imply
  \begin{equation*}
    Q_\pm = \frac{1}{2} \left(I_2 \pm \frac{1}{4}(\mathcal{R} + \mathcal{R}^*) \right) \pm \frac{1}{8} (\mathcal{R} - \mathcal{R}^*) = P_\pm \pm \frac{1}{8} (\mathcal{R} - \mathcal{R}^* + \mathcal{K})
  \end{equation*}
  and similarly
  \begin{equation*}
    Q_\pm^* = \frac{1}{2} \left(I_2 \pm \frac{1}{4} (\mathcal{R} + \mathcal{R}^*) \right) \mp \frac{1}{8} (\mathcal{R} - \mathcal{R}^*) = P_\pm \mp \frac{1}{8} (\mathcal{R} - \mathcal{R}^* - \mathcal{K}).
  \end{equation*}
  Hence, the projections $P_\pm$ constructed in Theorem~\ref{proposition_anti_commutor_R_3d} coincide, up to the map $\pm \frac{1}{8} (\mathcal{R} - \mathcal{R}^* + \mathcal{K})$, which is by Proposition~\ref{proposition_properties_R_3d} bounded from $H^{s-1}(\Sigma; \mathbb{C}^{2})$ to $H^{s}(\Sigma; \mathbb{C}^{2})$ and hence compact in $H^s(\Sigma; \mathbb{C}^2)$ for any $s \in \mathbb{R}$, with the skew projections onto Hardy spaces considered in \cite[Section~4.2]{AMSV22}. However, in contrast to $P_\pm$ the projections $Q_\pm$ are not self-adjoint in $L^2(\Sigma; \mathbb{C}^2)$, unless $\Sigma$ is a sphere. This is a drawback, as $Q_\pm Q_\mp^* \neq 0$, unless $\Sigma$ is a sphere, which makes the application of $Q_\pm$ in problems involving the Schur complement of $\mathcal{C}_z$ and related operators more involved.
\end{remark}

\subsection*{Data availability statement}

Data sharing not applicable to this article as no datasets were generated or analysed during the current study.

\subsection*{Competing Interests}

The author has no competing interests to declare that are relevant to the content of this article.

\end{document}